\documentclass[11pt,a4paper,reqno]{amsart}
\usepackage[T1]{fontenc}
\usepackage{lmodern,amsmath,amssymb,amsthm,mathtools}
\usepackage[margin=29mm]{geometry}
\usepackage[expansion=false]{microtype}
\usepackage{needspace}
\usepackage{tikz}
\usepackage[hidelinks]{hyperref}
\hypersetup{pdftitle={Absolutely minimising Lipschitz extensions between Hilbert spaces},pdfauthor={Krzysztof J. Ciosmak}}
\newtheorem{theorem}{Theorem}[section]
\newtheorem{lemma}[theorem]{Lemma}
\newtheorem{proposition}[theorem]{Proposition}
\newtheorem{corollary}[theorem]{Corollary}
\theoremstyle{remark}
\newtheorem{remark}[theorem]{Remark}
\newtheorem{example}[theorem]{Example}
\DeclareMathOperator{\Lip}{Lip}
\DeclareMathOperator{\dist}{dist}
\DeclareMathOperator{\conv}{Conv}
\DeclareMathOperator{\Span}{Span}
\newcommand{\Id}{\mathrm{Id}}
\DeclareMathOperator{\interior}{int}
\newcommand{\R}{\mathbb R}
\DeclareMathOperator{\cl}{cl}
\numberwithin{equation}{section}
\newcommand{\doi}[1]{\href{https://doi.org/#1}{\nolinkurl{#1}}}
\allowdisplaybreaks[1]

\title[Absolutely minimising extensions]{Absolutely minimising Lipschitz extensions between Hilbert spaces}
\author{Krzysztof J. Ciosmak}
\address{Beijing Institute of Mathematical Sciences and Applications,
No. 544, Hefangkou Village, Huaibei Town, Huairou District,
Beijing 101408, China}
\email{kciosmak@bimsa.cn}
\date{}
\begin{document}

\begin{abstract}
Every Lipschitz map from a closed subset of a real Hilbert space
into a real Hilbert space has an absolutely minimising Lipschitz
extension: an extension whose Lipschitz constant on every open
subset of the complement, bounded or not, equals the Lipschitz
constant of its restriction to the boundary of that subset.
Neither space need be separable. This answers a question of Naor
and Sheffield. For targets of dimension at least two we prove a
density theorem on connected domains: every optimal extension of
a boundary datum can be approximated, uniformly on compact sets,
by extensions whose Lipschitz constant on every nonempty open
subset, and on its boundary, equals the constant of the datum.
Consequently, in contrast with the scalar case, such a datum has
a unique absolutely minimising extension exactly when its optimal
extension is unique, and absolutely minimising extensions need
not be closed under uniform convergence, even for fixed smooth
data on a planar disk. The extensions are built by finite
piecewise isometric interpolation: on nested polyhedra with dense
union they are piecewise affine and preserve distances in at
least two directions on each simplex, a structure described
through leaves of $1$-Lipschitz maps. The boundary constraints
are handled by a midpoint argument for admissible values, and a
Kuratowski--Zorn argument removes the separability assumption
while preserving previously prescribed values. We also
characterise the bounded linear maps that are absolutely
minimising extensions of all data determining their operator
norm.
\end{abstract}
\maketitle

\section{Introduction}

The Kirszbraun--Valentine theorem extends a Lipschitz map between
real Hilbert spaces without increasing its Lipschitz constant
\cite{Kirszbraun,Valentine}. Absolute minimality imposes the same
requirement within every open region: the Lipschitz constant there
must be determined by the values on its boundary. We prove that
every Lipschitz datum on a closed subset of a Hilbert space admits
such an extension, with no separability or boundedness assumption.
When the target has dimension at least two, we prove more: the
extensions whose Lipschitz constant is attained on every open
subset of the domain are dense among all optimal extensions.
Bounded real-valued data have a unique absolutely minimising
extension; for targets of dimension at least two it is unique
exactly when the optimal extension is.

All Hilbert spaces below are real, with their norm distances.
For a map $w$ defined on a set containing $B$, write
\begin{equation*}
 \Lip(w,B)=\sup\left\{
 \frac{\|w(x)-w(y)\|}{\|x-y\|}\mid x,y\in B,\ x\ne y\right\},
\end{equation*}
with value zero when $B$ has at most one point. Let $A$ be closed
in a Hilbert space $X$, and let $u\colon A\to Y$ be Lipschitz.
An extension $\widetilde u\colon X\to Y$ is \emph{optimal} if
$\Lip(\widetilde u,X)=\Lip(u,A)$. It is an
\emph{absolutely minimising Lipschitz extension}, or AMLE, if it
is optimal and
\begin{equation}\label{eq:amle}
 \Lip(\widetilde u,V)=\Lip(\widetilde u,\partial V)
 \text{ for every nonempty open }V\subset X\setminus A.
\end{equation}
For data on the boundary of an open set $\Omega$, an AMLE is
defined on $\cl\Omega$ and satisfies this comparison for every
nonempty open $V\subset\Omega$. Boundaries are taken in the
source space unless a subspace is specified.

The comparison sets may be disconnected or unbounded. Restricting
them to relatively compact open sets would make the definition
vacuous in infinite-dimensional Hilbert spaces. Even in finite
dimensions, allowing unbounded sets imposes additional conditions,
as the scalar maximum principle and linear examples below show.

\begin{theorem}[Existence]\label{thm:closed}
Let $X,Y$ be real Hilbert spaces, let $A\subset X$ be closed,
and let $u\colon A\to Y$ be Lipschitz. Then $u$ has an AMLE
$\widetilde u\colon X\to Y$.
\end{theorem}

There are no boundedness or separability assumptions. This answers
Question~2 of Naor and Sheffield \cite[Section~1.4]{NaorSheffield}
on an absolutely minimising version of Kirszbraun's theorem.
The scalar case follows from Milman's theorem \cite{Milman}; our
construction treats targets of dimension at least two.

In Hilbert spaces, \eqref{eq:amle} is equivalent to the competitor
formulation used in \cite{NaorSheffield}:
\begin{equation*}
 \Lip(\widetilde u,V)\leq\Lip(v,V)
 \text{ whenever }v\colon X\to Y\text{ is Lipschitz and }
 v=\widetilde u\text{ on }X\setminus V.
\end{equation*}
Competitors agree on $\partial V$ by continuity. Conversely,
Kirszbraun's theorem extends the boundary values with their own
Lipschitz constant, and Lemma~\ref{lem:pasting} pastes that
extension into $V$ without increasing the global bound.

Our second theorem gives a stronger form of absolute minimality.
An $L$-Lipschitz map $\widetilde u\colon\cl\Omega\to Y$ is
\emph{saturated with constant $L$ on $\Omega$} if
\begin{equation}\label{eq:saturated}
 \Lip(\widetilde u,V)=\Lip(\widetilde u,\partial V)=L
 \text{ for every nonempty open }V\subset\Omega.
\end{equation}
Thus a saturated optimal extension is an AMLE, and its Lipschitz
constant cannot be reduced by restricting to any nonempty open set.

\begin{theorem}[Density of saturated extensions]\label{thm:domain}
Let $X,Y$ be real Hilbert spaces with $\dim Y\geq2$.
Let $\Omega\subsetneq X$ be nonempty, connected and open,
let $u\colon\partial\Omega\to Y$ be Lipschitz, and put
$L=\Lip(u,\partial\Omega)$. Given an $L$-Lipschitz extension
$v_0\colon\cl\Omega\to Y$ of $u$, a nonempty compact set
$Q\subset\cl\Omega$, and $\varepsilon>0$, there is an
$L$-Lipschitz extension $\widetilde u\colon\cl\Omega\to Y$
of $u$, saturated with constant $L$ on $\Omega$, such that
\begin{equation*}
 \sup\{\|\widetilde u(x)-v_0(x)\|\mid x\in Q\}<\varepsilon.
\end{equation*}
\end{theorem}

Saturated extensions are therefore dense among optimal extensions
for uniform convergence on compact sets. If $X$ is
finite-dimensional and $\Omega$ is bounded, the approximation is
uniform on $\cl\Omega$. In a one-dimensional source the theorem
reduces to affine interpolation on an interval or a constant
extension on a ray; the construction is needed when $\dim X\geq2$.

The density theorem has two immediate consequences. On a connected
domain, an AMLE is unique exactly when the optimal extension is
unique (Corollary~\ref{cor:uniqueness}). Moreover, AMLEs need not
be closed under uniform convergence: Example~\ref{ex:uniform-nonclosure}
gives fixed smooth boundary data on a planar disk and AMLEs
converging uniformly to an optimal extension that is not absolutely
minimising. For an arbitrary closed prescribed set, the closure
is described by optimality on each complementary component
(Corollary~\ref{cor:density}).

For bounded scalar Lipschitz data on a nonempty prescribed set,
the AMLE is unique even when the optimal extension is not, by Peres, Schramm,
Sheffield and Wilson \cite{PSSW}. Hence the density theorem fails
for one-dimensional targets (Remark~\ref{rem:dichotomy}). Scalar
AMLEs also obey a maximum principle with comparison over all open
sets, whereas vector-valued AMLEs need not remain in the closed
convex hull of the prescribed values
(Proposition~\ref{prop:maximum}).

The role of unbounded comparison sets is particularly explicit
for linear maps. Every bounded linear map satisfies comparison on
bounded open sets. A nonzero bounded linear map $T$ is an AMLE for
every closed datum that determines its operator norm precisely
when, for every $\varepsilon>0$, some two-dimensional subspace
satisfies $\|Te\|\geq(\|T\|-\varepsilon)\|e\|$ for all its
vectors (Proposition~\ref{prop:linear}). This includes maps that
strictly decrease every nonzero distance; one such AMLE preserves
no distance at all: its leaves, the maximal sets on which it
preserves distances \cite{CiosmakLeaves}, are all singletons
(Example~\ref{ex:strict-linear}).

The scalar theory begins with Aronsson's existence theorem
\cite{Aronsson} and Jensen's uniqueness theorem \cite{Jensen}
on bounded Euclidean domains. Comparison with cones was developed
by Crandall, Evans and Gariepy \cite{CEG}; see \cite{ACJ} for a
survey. Milman \cite{Milman}, Juutinen \cite{Juutinen}, and Peres,
Schramm, Sheffield and Wilson \cite{PSSW} treat extensions on
metric or length spaces, with the bounded-data theorem of
\cite{PSSW} using the all-open-set comparison considered here.
For finite metric-tree targets, Naor and Sheffield proved existence and
uniqueness on locally compact length spaces \cite{NaorSheffield}.
For Euclidean targets, Sheffield and Smart exhibited nonuniqueness
and introduced the stronger notion of a tight extension
\cite{SheffieldSmart}. Related vector-valued $L^\infty$
variational problems and Aronsson systems are studied in
\cite{BarronJensenWang,Katzourakis}.

Hirn and Le Gruyer established approximate absolute minimality
in a general extension framework on a compact working domain
\cite{HirnLeGruyer}; Theorem~\ref{thm:domain} instead imposes
exact absolute minimality and approximates a given extension on
compact sets. Continuity and uniform-distance constraints for
Hilbert-space extensions are studied in
\cite{CiosmakContinuity,CiosmakMonotone,CiosmakUniform2026}; in
particular, \cite{CiosmakUniform2026} characterises the reference
maps for which an extension can preserve a prescribed
uniform-distance bound. Our construction uses geometry studied by
Kirchheim, Spadaro and Sz\'ekelyhidi \cite[Sections~2 and~4]{KSS}:
for compact data in equidimensional Euclidean spaces, they
describe the hull forced by isometric subsets, construct locally
strict extensions outside it, and prove genericity of extensions
preserving lengths of rectifiable curves outside that hull.

We outline the proof of Theorem~\ref{thm:domain} for separable
sources, normalising the boundary Lipschitz constant to one. The
construction of Section~\ref{sec:separable} prescribes values at
finitely many new points at a time and extends them by Brehm's
theorem \cite{Brehm}, never changing values already prescribed.
It produces nested finite-dimensional polyhedra, with dense union
in each component of the domain minus a closed exceptional set,
on whose simplices the extension is affine isometric. Retaining
the finitely many vertex values forces every later extension to
preserve each old simplex, so no convergence of the intermediate
maps is needed, and a finite net controls the approximation on
the prescribed compact set. When the target has fewer dimensions
than the source, we first enlarge the target and then project,
retaining two isometric directions on every full-dimensional
simplex. The simplices are convex subsets of leaves of the
auxiliary map; Remark~\ref{rem:structure} records the resulting
geometry.

The construction must respect all boundary constraints, including
those that leave no room for perturbation. At each source point,
the admissible extension values form an intersection of closed
balls; the \emph{degeneracy set} $D$ consists of the points where
this intersection has empty interior, and it is the exceptional
set above. If $D$ is proper, a reflection and convexity argument
shows that all values on $D$ are already forced by the original
datum (Lemma~\ref{lem:forced-degeneracy}), so the initial map is
retained there automatically. Outside $D$, finitely many new
values can be chosen with a positive uniform margin against that
datum, and a small common translation of these target values,
followed by convex averaging, recovers such a margin on each
compact region to be enlarged (Lemma~\ref{lem:addition}).

Saturation is proved by contradiction. Suppose a comparison set
$V$ has boundary Lipschitz constant less than one. Pasting a
strictly contracting extension into $V$ gives an extension that
is strictly contracting on $V$, and $V$ avoids $D$. On a simplex
whose linear part is isometric in two directions, the boundary of
$V$ meets each plane parallel to those directions in at most one
point, and removing a point does not disconnect a convex planar
region; projection onto the orthogonal complement, followed by
pasting across facets, shows that this boundary cannot separate a
polyhedron carrying such simplices
(Lemma~\ref{lem:nonseparation}). This is where the
target-dimension assumption enters. Since the polyhedra are
nested with dense union, $V$ contains a dense subset of an entire
component of the domain outside $D$, and the pasted extension is
strictly contracting on that component. Such a component cannot
touch $D$ inside the domain (Lemma~\ref{lem:contact}): reflecting
a point of a contracting ball through a nearby degenerate point
and averaging the two extension values, the strict contraction
supplies a linear margin while the parallelogram identity bounds
the error quadratically, so the average would be an interior
admissible value at the degenerate point. Thus the component is
the whole connected domain, contradicting the boundary Lipschitz
constant.

For nonseparable sources, we choose sections on which the visible
boundary data impose all the original constraints. A saturated map
on a section can be enlarged without changing its old values:
adding two directions orthogonal to the section keeps the new
region connected, and the theorem for smaller source density
fills it. Boundary inequalities and saturation pass to the closure
of a coherent union. The Kuratowski--Zorn lemma assembles a
maximal family of these maps, and a source-density bound permits
one more section to be added unless the family already covers a
basis of the whole source.

Sections~\ref{sec:admissible}--\ref{sec:hilbert-tangent}
establish the geometry of admissible values needed in
Section~\ref{sec:separable}, which proves the density theorem
for separable sources; Section~\ref{sec:nonseparable} extends it
to arbitrary Hilbert sources. Section~\ref{sec:existence} passes to
arbitrary prescribed sets. Sections~\ref{sec:consequences}
and~\ref{sec:linear} contain the uniqueness results, the
uniform-limit example, and the linear characterisations.

\section{Preliminaries}\label{sec:preliminaries}

We use Kirszbraun's theorem to extend compatible data and Brehm's
theorem to create isometric pieces. Two elementary facts connect
these operations: extensions can be pasted across matching
boundaries, and distance-preserving data determine all values on
their convex hull.

\subsection{Notation and geometric conventions}
Hilbert spaces are real unless stated otherwise. We write
$\|\cdot\|$ for both the Hilbert norm and the operator norm,
$\langle\cdot,\cdot\rangle$ for the inner product, and $P_E$
for orthogonal projection onto a closed linear subspace $E$.
The symbols $\conv$, $\operatorname{Aff}$ and
$\Span$ denote convex, affine and linear hulls;
$\cl$ denotes norm closure. The inclusion $S\subset T$ allows
equality, whereas $S\subsetneq T$ does not. A subset of $X$ is \emph{proper} if it is not all of $X$. Hilbert dimension is
the cardinality of an orthonormal basis, and $\operatorname{dens}X$
is the least cardinality of a norm-dense subset of $X$.

The open ball of radius $r$ about $x$ in $X$ is $B_r^X(x)$;
we omit $X$ when it is clear. Subscripts on
$\operatorname{int}_E$, $\partial_E$ and $\cl_E$ indicate that
interior, boundary and closure are taken in the specified affine
subspace $E$. This distinction matters for finite-dimensional
polyhedra in an infinite-dimensional source. All other boundaries
and interiors are taken in the ambient Hilbert space. Components
are connected components; those of an open set are open and
polygonally connected.

A map $w$ on an
open set $O\subset X$ is \emph{locally strictly $1$-Lipschitz} if, for
every $x\in O$, there are $r>0$ and $0\leq k<1$ such that
$B_r^X(x)\subset O$ and $\Lip(w,B_r^X(x))\leq k$.
The numbers $r$ and $k$ may depend on $x$.

A \emph{polytope} is the convex hull of finitely many points;
a \emph{compact polyhedral set} is a finite union of polytopes
and need not be convex. Triangulations consist of closed simplices
meeting along common faces. If the triangulated set is the closure
of its interior in $E$, its full-dimensional simplices cover it and
every other simplex is a face of one of them. A map is \emph{piecewise affine} if
it is continuous and affine on every simplex of a triangulation.
A simplex in $E$ is \emph{full-dimensional} if its dimension is
$\dim E$. A triangulation is \emph{locally finite} if every point
has a neighbourhood meeting only finitely many simplices.

An \emph{affine isometry} is a map $x\mapsto a+Tx$ whose linear
part $T$ is an isometric embedding; surjectivity is not required.
A bounded linear map $T\colon E\to F$ is a \emph{coisometry}
if $TT^*=\Id_F$. Equivalently, its restriction to
$(\ker T)^\perp$ is an isometric isomorphism onto $F$.
When $F$ is a subspace of a larger target, we specify that $T$
is a coisometry \emph{onto $F$}. Statements about the linear part
of an affine map use these conventions.

\subsection{Two extension theorems}
We use the following classical Hilbert-space extension theorem;
see \cite{Kirszbraun,Valentine} and
\cite[Theorem~1.2]{AzagraLeGruyerMudarra}.

\begin{theorem}[Kirszbraun--Valentine]\label{thm:kirszbraun}
Let $X,Y$ be real Hilbert spaces and let $A\subset X$.
Every Lipschitz map $u\colon A\to Y$ has an extension
$\widetilde u\colon X\to Y$ satisfying
\begin{equation*}
 \Lip(\widetilde u,X)=\Lip(u,A).
\end{equation*}
\end{theorem}

Neither separability nor closedness of $A$ is required. The second
result supplies additional structure when the datum is finite.
We use the locally finite form of Brehm's interpolation theorem
\cite{Brehm}; see \cite{AkopyanTarasov} for a constructive proof.

\begin{theorem}[Brehm]\label{thm:brehm}
Let $E,F$ be Euclidean spaces of the same finite dimension, let
$P\subset E$ be finite, and let $u\colon P\to F$ be $1$-Lipschitz.
There is a $1$-Lipschitz extension $\widetilde u\colon E\to F$ and a
locally finite triangulation of $E$ such that the restriction of
$\widetilde u$ to each simplex is the restriction of an affine
isometry from $E$ onto $F$.
\end{theorem}

In particular, the restriction of $\widetilde u$ to any compact
polyhedral set in $E$ admits a finite triangulation on which it
is affine and isometric on each simplex. Indeed, compactness
leaves only finitely many simplices of the global triangulation.
Their intersections with the given polytopes are polytopes, not
necessarily simplices, on which $\widetilde u$ is affine;
triangulating these polytopes compatibly along their common faces
gives the assertion.

The same conclusion, with affine isometric embeddings, holds
when $\dim E\leq\dim F<\infty$. Embed $E$ in a Euclidean
space $\widehat E$ of dimension $\dim F$, apply the theorem,
and restrict to $E$. The intersections of the ambient simplices with $E$ are
polytopes on which the map is affine; triangulating them
compatibly gives the required locally finite triangulation.

\subsection{Pasting and boundary comparison}

We will repeatedly use the following consequences of the segment
structure of a Hilbert space. The open sets in the lemma need not
be bounded or connected, and the family in part~\textup{(i)} may
be uncountable.

\Needspace{6\baselineskip}
\begin{lemma}[Pasting Lipschitz maps]\label{lem:pasting}
Let $X,Y$ be Hilbert spaces and let $L\geq0$.
\begin{enumerate}
\item[\textup{(i)}] Let $v\colon X\to Y$ be $L$-Lipschitz, and let
$(U_i)_{i\in I}$ be pairwise disjoint open subsets of $X$.
For each $i\in I$, let $v_i\colon\cl U_i\to Y$ be
$L$-Lipschitz and agree with $v$ on $\partial U_i$.
The map equal to $v_i$ on each $U_i$ and to $v$ elsewhere
is $L$-Lipschitz on $X$.
\item[\textup{(ii)}] Let $V\subset X$ be open, and let
$w\colon\cl V\to Y$ be continuous. Suppose every point of $V$
has a neighbourhood in $V$ on which $w$ is $L$-Lipschitz, and
$\Lip(w,\partial V)\leq L$. Then $\Lip(w,\cl V)\leq L$.
\end{enumerate}
\end{lemma}
\begin{proof}
In both parts, split a segment $[x,y]$ into at most three
pieces. For part~\textup{(i)}, denote the pasted map by $w$.
The estimate is immediate if $x,y$ lie in the same $U_i$ or
both lie outside their union. Otherwise, let $x'$ be the first
exit from the $U_i$ containing $x$, and let $y'$ be the last
entry into the $U_j$ containing $y$; use $x'=x$ or $y'=y$
when that endpoint lies outside the union. The end pieces are
controlled by the corresponding $v_i$, while the middle pair
is controlled by $v$, since boundary values agree.

For part~\textup{(ii)}, first take $x,y\in V$. Local bounds
control every segment contained in $V$, by a finite subdivision.
If $[x,y]$ leaves $V$, let $x',y'$ be its first and last points
outside $V$. Continuity extends the local bounds to the end
pieces, and the boundary bound controls the middle pair.

In either case, the estimates add to
\begin{equation*}
 \|w(x)-w(y)\|
 \leq L\bigl(\|x-x'\|+\|x'-y'\|+\|y'-y\|\bigr)
 =L\|x-y\|.
\end{equation*}
Continuity gives the conclusion on $\cl V$ in
part~\textup{(ii)}.
\end{proof}

In particular, an $L$-Lipschitz datum on a closed subset of $X$
may be extended separately on the complementary components, provided
each extension has Lipschitz constant at most $L$. To see this, first
take a global Kirszbraun extension and then apply part~\textup{(i)}.

\subsection{Affine rigidity and leaves}\label{subsec:leaves}
The construction of Section~\ref{sec:separable} will repeatedly
prescribe the values at the vertices of a simplex. If those values preserve
all distances, the $1$-Lipschitz bound forces the affine value at
every point of the simplex. We record this rigidity together with
its interpretation in terms of leaves. The definition and geometric
properties below agree with
\cite[Definition~2.2, Lemma~2.4 and Corollary~2.5]{CiosmakLeaves};
the short proof works in arbitrary Hilbert spaces.

Let $u\colon X\to Y$ be $1$-Lipschitz. A nonempty set
$\mathcal S\subset X$ is a \emph{leaf of $u$} if
$u|_{\mathcal S}$ preserves distances and, for every
$y\in X\setminus\mathcal S$, there is $x\in\mathcal S$ such that
\begin{equation*}
 \|u(y)-u(x)\|<\|y-x\|.
\end{equation*}
Equivalently, a leaf is maximal with respect to inclusion among
sets on which $u$ preserves distances. The same definition applies
to a $1$-Lipschitz map defined on a subset of $X$, in particular to
a datum on a finite set.

\begin{lemma}[Existence and geometry of leaves]\label{lem:leaves}
Let $X,Y$ be real Hilbert spaces and let $u\colon X\to Y$ be
$1$-Lipschitz. Every subset on which $u$ preserves distances is
contained in a leaf. Each leaf $\mathcal S$ is closed and convex,
and $u|_{\mathcal S}$ is the restriction of an affine isometric
embedding on $\cl(\operatorname{Aff}\mathcal S)$.
More precisely, for every $x_0\in\mathcal S$ there is a linear
isometric embedding
\begin{equation*}
 T\colon\cl\bigl(\Span(\mathcal S-x_0)\bigr)\to Y
\end{equation*}
such that
\begin{equation*}
 u(x)=u(x_0)+T(x-x_0)\text{ for }x\in\mathcal S.
\end{equation*}
\end{lemma}
\begin{proof}
A union of a chain of sets on which $u$ preserves distances
has the same property. The Kuratowski--Zorn lemma therefore
gives a maximal such set $\mathcal S$ containing any given
isometric subset. Its closure still preserves distances, so
$\mathcal S$ is closed. Fix $x_0\in\mathcal S$. Polarisation
shows that $x-x_0\mapsto u(x)-u(x_0)$, $x\in\mathcal S$,
extends linearly and continuously to an isometry $T$ on
$\cl(\Span(\mathcal S-x_0))$.

For $x,y\in\mathcal S$ and $z=(1-t)x+ty$, $0\leq t\leq1$,
equality in the triangle inequality forces
$u(z)=(1-t)u(x)+tu(y)$. This is the value of the same affine
isometry, so $\mathcal S\cup\{z\}$ still preserves distances.
Maximality gives $z\in\mathcal S$, proving convexity.
\end{proof}

We refer to the following consequence as \emph{affine rigidity}:
whenever prescribed values preserve distances on a set $S$, every
$1$-Lipschitz extension agrees with their affine isometric
interpolation on $\cl(\conv S)$. Thus finitely many
values can permanently determine an entire simplex. Distinct
leaves can intersect only along their relative boundaries
\cite[Lemma~2.11]{CiosmakLeaves}.

\section{Admissible values and uniform margins}\label{sec:admissible}

We will consider a construction that modifies an extension while
keeping an arbitrary closed set of prescribed values fixed. To make such a modification,
we need a positive margin in the Lipschitz inequalities, uniform
over the prescribed set. Strict inequality for each prescribed
point separately would not suffice, since the set need not be
compact. This section identifies where a uniform margin is
available and shows how local strictness supplies it.

Throughout this section, $X$ and $G$ are Hilbert spaces with
$G\ne\{0\}$, $A\subset X$ is nonempty and closed, and
$u\colon A\to G$ is $1$-Lipschitz. The target is called $G$
rather than $Y$ because the results will be applied both to $Y$
and to the auxiliary target of Section~\ref{sec:separable},
which enlarges $Y$ when it has too few dimensions for Brehm's
theorem. For $x\in X$, define the
\emph{admissible-value set}
\begin{equation*}
 \mathcal P(u,A,x)=\{p\in G\mid
 \|p-u(y)\|\leq\|x-y\|\text{ for every }y\in A\}.
\end{equation*}
Kirszbraun's theorem says precisely that this set is nonempty
and that every $p\in\mathcal P(u,A,x)$ occurs as the value at
$x$ of a global $1$-Lipschitz extension of $u$. The set is bounded,
closed and convex. For $x\notin A$, it is the intersection of
the closed balls $\cl B_{\|x-y\|}(u(y))$, $y\in A$; if
$x\in A$, it is the singleton $\{u(x)\}$.

To measure the room left by all the constraints at once, put
\begin{equation}\label{eq:margin-function}
 \Phi(x,p)=\Phi_{u,A}(x,p)
 =\sup\{\|p-u(y)\|-\|x-y\|\mid y\in A\}.
\end{equation}
Thus $\Phi(x,p)\leq-\eta$ means that the proposed value $p$
has margin at least $\eta$ against every prescribed value.
The function is finite: for a fixed $y_0\in A$,
\begin{equation*}
 \|p-u(y_0)\|-\|x-y_0\|
 \leq\Phi(x,p)\leq\|p-u(y_0)\|+\|x-y_0\|.
\end{equation*}
It is convex in $p$ and satisfies
\begin{equation*}
 |\Phi(x,p)-\Phi(x',p')|
 \leq\|x-x'\|+\|p-p'\|.
\end{equation*}
For $\eta>0$, the uniform margin has the equivalent geometric form
\begin{equation*}
 \Phi(x,p)\leq-\eta
 \quad\Longleftrightarrow\quad
 \cl B_\eta(p)\subset\mathcal P(u,A,x).
\end{equation*}
Indeed, a ball of radius $\eta$ about $p$ lies in
$\cl B_r(q)$ exactly when $\|p-q\|+\eta\leq r$.
In particular,
\begin{equation}\label{eq:margin-interior}
 \mathcal P(u,A,x)=\{p\in G\mid\Phi(x,p)\leq0\},\qquad
 \interior_G\mathcal P(u,A,x)=\{p\in G\mid\Phi(x,p)<0\}.
\end{equation}

The \emph{degeneracy set} consists of the points at which no
admissible value has a positive margin:
\begin{equation*}
 D(u,A)=\{x\in X\mid\interior_G\mathcal P(u,A,x)=\emptyset\}
       =\{x\in X\mid\Phi(x,p)\geq0\text{ for every }p\in G\}.
\end{equation*}
This set is closed, as an intersection of the closed sets
$\{x\mid\Phi(x,p)\geq0\}$, and it contains $A$ since
$G\ne\{0\}$. Degeneracy refers to empty
interior in the target, and does not in general mean that the
admissible value is unique. For example, let $X=G=\R^2$, take
$A=\R\times\{0\}$, and prescribe $u(t,0)=(t,0)$. The
constraints at $x=(s,r)$ give
\begin{equation*}
 \mathcal P(u,A,(s,r))=\{s\}\times[-|r|,|r|].
\end{equation*}
Thus $D(u,A)=X$, although the admissible values are not unique
off $A$.

If the degeneracy set is proper, degeneracy does force uniqueness:
all its values are then determined by the original datum.

\begin{lemma}[Forced values on a proper degeneracy set]\label{lem:forced-degeneracy}
If $D(u,A)\ne X$, then $\mathcal P(u,A,z)$ is a singleton
for every $z\in D(u,A)$. Consequently, all global
$1$-Lipschitz extensions of $u$ agree on $D(u,A)$.
\end{lemma}
\begin{proof}
Choose $x_0\notin D(u,A)$ and $p_0\in G$ with
$\Phi(x_0,p_0)\leq-\eta<0$. Suppose $z\in D(u,A)$ has two
distinct admissible values $b,c$, and choose an admissible value
$q$ at the reflected point $2z-x_0$. Fix $y\in A$.
The parallelogram identity in the target gives
\begin{equation*}
 \Bigl\|\frac{b+c}{2}-u(y)\Bigr\|
 \leq\sqrt{\|z-y\|^2-\tfrac14\|b-c\|^2}.
\end{equation*}
The margin at $p_0$, the admissibility of $q$, and the
parallelogram identity in the source give
\begin{equation}\label{eq:reflected-margin}
 \Bigl\|\frac{p_0+q}{2}-u(y)\Bigr\|
 \leq\frac{\|x_0-y\|+\|2z-x_0-y\|}{2}-\frac\eta2
 \leq\sqrt{\|z-y\|^2+\|x_0-z\|^2}-\frac\eta2.
\end{equation}
Put
\begin{equation*}
 \lambda=\frac{\|b-c\|^2}{\|b-c\|^2+4\|x_0-z\|^2}\in(0,1],
 \qquad
 p=(1-\lambda)\frac{b+c}{2}+\lambda\frac{p_0+q}{2}.
\end{equation*}
Concavity of the square root yields, uniformly over $y\in A$,
\begin{equation*}
 \begin{aligned}
 \|p-u(y)\|
 &\leq(1-\lambda)\sqrt{\|z-y\|^2-\tfrac14\|b-c\|^2}
       +\lambda\sqrt{\|z-y\|^2+\|x_0-z\|^2}-\frac{\lambda\eta}2\\
 &\leq\sqrt{\|z-y\|^2-\tfrac14(1-\lambda)\|b-c\|^2
       +\lambda\|x_0-z\|^2}-\frac{\lambda\eta}2
 =\|z-y\|-\frac{\lambda\eta}2.
 \end{aligned}
\end{equation*}
This positive margin contradicts $z\in D(u,A)$.
\end{proof}

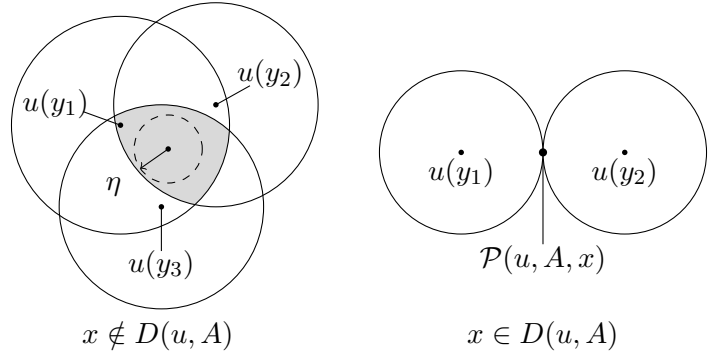
\begin{figure}[t]
\centering
\begin{tikzpicture}[scale=0.9]
 \begin{scope}
  \coordinate (c1) at (0,0); \coordinate (c2) at (1.4,0.3); \coordinate (c3) at (0.6,-1.2);
  \begin{scope}
   \clip (c1) circle (1.6); \clip (c2) circle (1.5);
   \fill[gray!30] (c3) circle (1.5);
  \end{scope}
  \draw[thin] (c1) circle (1.6); \draw[thin] (c2) circle (1.5); \draw[thin] (c3) circle (1.5);
  \fill (c1) circle (0.04); \fill (c2) circle (0.04); \fill (c3) circle (0.04);
  \node at (-0.95,0.32) {$u(y_1)$}; \draw[thin] (-0.55,0.2) -- (c1);
  \node at (2.2,0.75) {$u(y_2)$}; \draw[thin] (1.9,0.58) -- (c2);
  \node at (0.6,-2.05) {$u(y_3)$}; \draw[thin] (0.6,-1.85) -- (c3);
  \coordinate (p) at (0.7,-0.35);
  \draw[dashed] (p) circle (0.5); \fill (p) circle (0.04);
  \draw[->,thin] (p) -- ++(215:0.5); \node at (-0.1,-0.9) {$\eta$};
  \node at (0.55,-3.1) {$x\notin D(u,A)$};
 \end{scope}
 \begin{scope}[xshift=6.2cm]
  \coordinate (d1) at (-1.2,-0.4); \coordinate (d2) at (1.2,-0.4);
  \draw[thin] (d1) circle (1.2); \draw[thin] (d2) circle (1.2);
  \fill (d1) circle (0.04); \fill (d2) circle (0.04);
  \node at (-1.2,-0.7) {$u(y_1)$}; \node at (1.2,-0.7) {$u(y_2)$};
  \fill (0,-0.4) circle (0.06); \draw[thin] (0,-0.5) -- (0,-1.75); \node at (0,-1.98) {$\mathcal P(u,A,x)$};
  \node at (0,-3.1) {$x\in D(u,A)$};
 \end{scope}
\end{tikzpicture}
\caption{Admissible values are an intersection of balls in the
target. A uniform margin $\eta>0$ at $p$ means that the whole
ball $\cl B_\eta(p)$ is admissible (left). At a degenerate source
point the admissible set has empty interior; a singleton is shown
on the right. Such a set can have more than one point only when
$D(u,A)=X$ (Lemma~\ref{lem:forced-degeneracy}).}
\label{fig:admissible}
\end{figure}

Passing to a larger nonzero Hilbert target does not change
$D(u,A)$. Indeed, if $u(A)$ lies in a nonzero closed subspace
$Y\subset G$, orthogonal projection gives
$\Phi(x,P_Yp)\leq\Phi(x,p)$. A value with positive margin
therefore exists in $G$ exactly when one exists in $Y$. We will
use this when enlarging the target in
Section~\ref{sec:separable}.

For a chosen extension $v$, write $\sigma_v(x)=-\Phi(x,v(x))$
for the margin of its actual value at $x$. The next lemma turns
a local Lipschitz constant below one into a positive lower bound
for this margin, uniform over all of $A$.

\begin{lemma}[A uniform margin from local strictness]\label{lem:gap}
Let $v\colon X\to G$ be a global $1$-Lipschitz extension of $u$.
The function
\begin{equation*}
 \sigma_v(x)=-\Phi(x,v(x))
 =\inf\{\|x-y\|-\|v(x)-u(y)\|\mid y\in A\}
\end{equation*}
is finite, nonnegative and $2$-Lipschitz.
It is positive at $x$ exactly when
$v(x)\in\interior_G\mathcal P(u,A,x)$. If $r>0$,
$B_r(x)\cap A=\emptyset$, and $\Lip(v,B_r(x))\leq k<1$, then
\begin{equation*}
 \sigma_v(x)\geq(1-k)r.
\end{equation*}
Consequently, if $v$ is locally strictly $1$-Lipschitz on an
open set $O\subset X\setminus A$, then $\sigma_v>0$ on $O$, and
\begin{equation*}
 \min\{\sigma_v(x)\mid x\in K\}>0
\end{equation*}
for every nonempty compact set $K\subset O$.
\end{lemma}
\begin{proof}
The assertions about $\sigma_v$ follow from the corresponding
properties of $\Phi$, the extension property, and
\eqref{eq:margin-interior}. For the quantitative bound, fix
$y\in A$ and let $z\in[x,y]$ satisfy $\|z-x\|=r$.
On the first part of this segment the extension contracts by a
factor at most $k$; on the rest it is $1$-Lipschitz. Thus
\begin{equation*}
 \|v(x)-u(y)\|\leq kr+\|z-y\|
 =\|x-y\|-(1-k)r,
\end{equation*}
where continuity allows the local estimate up to the sphere.
Taking the infimum over $y$ gives the uniform margin. The last
assertion follows from positivity and continuity on $K$.
\end{proof}

The estimate has a useful local interpretation. At a point of
$D(u,A)\setminus A$, every global $1$-Lipschitz extension has
Lipschitz constant one on every neighbourhood. Outside $D(u,A)$,
a value with margin $\eta>0$ may be prescribed constantly on
$\cl B_\eta(x)$: for $z$ in that ball,
$\|p-u(y)\|\leq\|x-y\|-\eta\leq\|z-y\|$.
Kirszbraun's theorem then gives a global extension with this
constant neighbourhood.

Finally, a comparison boundary with Lipschitz constant below one
provides a locally strict replacement inside it. Such a comparison
set must therefore avoid $D(u,A)$.

\begin{lemma}[Strict comparison avoids degeneracy]\label{lem:paste}
Let $\widetilde u\colon X\to G$ be a global $1$-Lipschitz
extension of $u$, and let $V\subset X\setminus A$ be nonempty
and open. If $k=\Lip(\widetilde u,\partial V)<1$, there is a
global $1$-Lipschitz extension $w$ of $u$ which agrees with
$\widetilde u$ outside $V$ and is $k$-Lipschitz on $V$.
Moreover, $V\cap D(u,A)=\emptyset$.
\end{lemma}
\begin{proof}
Since $A\ne\emptyset$, the set $V$ is proper and
$\partial V\ne\emptyset$. Extend
$\widetilde u|_{\partial V}$ to a global $k$-Lipschitz map by
Kirszbraun's theorem and paste it with $\widetilde u$ outside $V$.
Lemma~\ref{lem:pasting} gives the global $1$-Lipschitz bound.
Every point of $V$ has a neighbourhood on which the new extension
is $k$-Lipschitz, so Lemma~\ref{lem:gap} excludes degeneracy there.
\end{proof}

\section{Finite additions with a uniform gap}\label{sec:finite-additions}

The construction of Section~\ref{sec:separable} will repeatedly
prescribe values at finitely many new points while preserving all
previously prescribed values. Its next step needs a positive
margin against the original datum on a compact region. We obtain
this margin by translating all new target values by the same
small vector and averaging finitely many compatible extensions.
The empty finite datum is allowed and will initialise the construction.

\begin{lemma}[Preserving a margin under finite additions]\label{lem:addition}
Let $X,G$ be Hilbert spaces with $G\ne\{0\}$, let
$A\subset X$ be nonempty and closed, and let
$u\colon A\to G$ be $1$-Lipschitz. Let $P\subset X$ be finite,
possibly empty, and let $a\colon P\to G$ be $1$-Lipschitz.
Suppose that for some $\delta>0$,
\begin{equation}\label{eq:cross-gap}
 \|a(p)-u(y)\|\leq\|p-y\|-\delta
 \quad(p\in P,\ y\in A),
\end{equation}
a condition that is vacuous when $P=\emptyset$.
For every nonempty compact set $K\subset X\setminus D(u,A)$,
the combined datum on $A\cup P$ has a global $1$-Lipschitz
extension $w\colon X\to G$ and a number $\gamma>0$ such that
\begin{equation*}
 \|w(x)-u(y)\|\leq\|x-y\|-\gamma
 \quad(x\in K,\ y\in A).
\end{equation*}
\end{lemma}
\begin{proof}
Fix $x\notin D(u,A)$ and put $B=\mathcal P(u,A,x)$.
We first find a value in $\interior_G B$ compatible with the
new data. This is immediate when $P=\emptyset$.
Otherwise, put $Q=\mathcal P(a,P,x)$. For every $e\in G$
with $\|e\|\leq\delta$, the datum $u$ on $A$ and $a+e$
on $P$ is $1$-Lipschitz: translation preserves the new mutual
distances, and
\begin{equation*}
 \|a(p)+e-u(y)\|\leq\|p-y\|-\delta+\|e\|\leq\|p-y\|.
\end{equation*}
Kirszbraun's theorem therefore gives $B\cap(Q+e)\ne\emptyset$.
Choose $b_0\in\interior_G B$, $q_0\in Q$, and $t>0$ so
small that $e=t(q_0-b_0)$ has norm at most $\delta$.
There are $z\in B$ and $q\in Q$ with $z=q+e$, and
\begin{equation*}
 b=\frac{z+t b_0}{1+t}=\frac{q+t q_0}{1+t}
 \in\interior_G B\cap Q.
\end{equation*}
Thus in either case we may extend the combined datum and an
interior value at $x$ by Kirszbraun's theorem. Denote this
extension by $W_x$. Its margin against $u$ is positive at $x$,
and hence on a neighbourhood of $x$.

By compactness, finitely many maps $W_1,\ldots,W_m$ have
such neighbourhoods covering $K$. Their average $w$ extends the
combined datum and, by convexity of $\Phi$, satisfies
\begin{equation*}
 w=\frac1m\sum_{j=1}^mW_j,
 \qquad \sigma_w\geq\frac1m\sum_{j=1}^m\sigma_{W_j}>0
 \quad\text{on }K.
\end{equation*}
Take $\gamma=\min\{\sigma_w(x)\mid x\in K\}>0$.
\end{proof}

\section{Contact with the degeneracy set}\label{sec:hilbert-tangent}

Let $w$ be a global $1$-Lipschitz extension of $u$, and let
$\sigma_w$ be the margin of its actual values, as in
Lemma~\ref{lem:gap}. This section compares two bounds on
$\sigma_w$ near a point $z\in D(u,A)\setminus A$.
The first is an upper bound, quadratic in $\|x-z\|$: it comes
from the reflection argument in the proof of
Lemma~\ref{lem:forced-degeneracy}, since a larger margin of
$w(x)$ would make the midpoint of $w(x)$ and $w(2z-x)$ a value
with positive margin at $z$, contradicting degeneracy.
The second is a lower bound, linear in the radius: by
Lemma~\ref{lem:gap}, if $w$ has Lipschitz constant $k<1$ on a
ball $B_r(x)$ disjoint from $A$, then $\sigma_w(x)\geq(1-k)r$.
For $x$ close to $z$ and $r$ comparable to $\|x-z\|$, the linear
bound exceeds the quadratic one. Consequently, a component of
the complement of $D(u,A)$ on which $w$ is strictly contracting
cannot have a boundary point outside $A$
(Lemma~\ref{lem:contact}).

\begin{lemma}[Quadratic bound on the margin near a degenerate point]\label{lem:quadratic}
Let $X,G$ be Hilbert spaces with $G\ne\{0\}$, let
$A\subset X$ be nonempty and closed, and let $u\colon A\to G$
be $1$-Lipschitz. Every global $1$-Lipschitz extension $w$ of
$u$ satisfies
\begin{equation}\label{eq:quadratic-actual-margin}
 \sigma_w(x)\leq\frac{\|x-z\|^2}{\dist(z,A)}
 \qquad(x\in X,\ z\in D(u,A)\setminus A).
\end{equation}
\end{lemma}
\begin{proof}
Put $p=(w(x)+w(2z-x))/2$. The reflected-value estimate
\eqref{eq:reflected-margin}, with margin $\sigma_w(x)$ at $x$,
gives, for $y\in A$ and $r=\|z-y\|$,
\begin{equation*}
 \|p-u(y)\|
 \leq\sqrt{r^2+\|x-z\|^2}-\frac{\sigma_w(x)}2
 \leq r+\frac{\|x-z\|^2}{2\dist(z,A)}-\frac{\sigma_w(x)}2.
\end{equation*}
If \eqref{eq:quadratic-actual-margin} failed, $p$ would have a
positive uniform margin at $z$, contradicting degeneracy.
\end{proof}

Near a contact point outside $A$ between $D(u,A)$ and a
contracting component, balls with radius equal to their distance
from $D(u,A)$ would supply a linear margin that the quadratic
estimate cannot accommodate.

\begin{lemma}[Contact with the degeneracy set]\label{lem:contact}
Let $X,G$ be Hilbert spaces with $G\ne\{0\}$, let
$A\subset X$ be nonempty and closed, let $u\colon A\to G$ be
$1$-Lipschitz. Let
$\Omega\subset X\setminus A$ be open with $\partial\Omega\subset A$,
and let $U$ be a connected component of $\Omega\setminus D(u,A)$.
If a global $1$-Lipschitz extension $w$ of $u$ satisfies
$\Lip(w,U)<1$, then
\begin{equation*}
 \partial U\cap\Omega=\emptyset.
\end{equation*}
\end{lemma}
\begin{proof}
Put $k=\Lip(w,U)<1$ and suppose $q\in\partial U\cap\Omega$.
Then $q\in D(u,A)\setminus A$. Choose $x_n\in U$ tending to $q$,
put $r_n=\dist(x_n,D(u,A))>0$, and take $z_n\in D(u,A)$ with
$\|x_n-z_n\|<2r_n$. We have $r_n\to0$ and $z_n\to q$.
The connected ball $B_{r_n}(x_n)$ contains $x_n$ and avoids
$D(u,A)\supset\partial\Omega$, so it lies in $U$.
For all sufficiently large $n$, Lemma~\ref{lem:gap} and
\eqref{eq:quadratic-actual-margin} give
\begin{equation*}
 (1-k)r_n\leq\sigma_w(x_n)
 <\frac{4r_n^2}{\dist(z_n,A)}.
\end{equation*}
Dividing by $r_n$ contradicts $r_n\to0$ and
$\dist(z_n,A)\to\dist(q,A)>0$.
\end{proof}

\section{Polyhedral approximation in separable sources}\label{sec:separable}

In this section we present the construction that proves
Theorem~\ref{thm:domain} for separable sources. Starting from a
given $1$-Lipschitz extension, it produces a new extension that
is close to the given one on a prescribed compact set and whose
values, in every component of $X\setminus D(u,A)$, are affine
isometric simplex by simplex on nested connected polyhedra with
dense union. The construction proceeds in finitely many steps at a
time: each step prescribes the values at the vertices of finitely
many new simplices, and values already prescribed are never
changed. By affine rigidity (Lemma~\ref{lem:leaves}), fixing the
finitely many vertex values is enough to make every later
extension affine isometric on the simplex, and
Lemma~\ref{lem:addition} preserves a positive margin against the
original datum, so that a further step is always possible.

We first explain why extensions of this form admit boundary
comparison. We then isolate the finite enlargement step and
iterate it. Throughout,
interiors of polyhedra are taken in their finite-dimensional source
subspaces; they need not be open in $X$.

\subsection{Why two isometric directions suffice}

We show that a closed set on which the map is strictly
contracting cannot separate a polyhedron whose simplices carry
affine isometric values in two directions. The reason is
topological: removing one point does not disconnect an open
convex subset of a plane. Orthogonal projection applies this to
all parallel planes at once; pasting across facets passes from a
simplex to a polyhedron.

\begin{lemma}[Nonseparation by sparse fibres]\label{lem:fibre-nonseparation}
Let $H$ be a Hilbert space, let $W\subset H$ be two-dimensional,
and let $O\subset H$ be nonempty, open and connected.
If $S\subset O$ is relatively closed and meets each affine plane
parallel to $W$ in at most one point, then $O\setminus S$ is
connected and dense in $O$.
\end{lemma}
\begin{proof}
A ball meets some affine plane parallel to $W$ in a disk, so $S$
contains no ball; this gives density.

First let $O$ be a ball. Since $S$ is relatively closed,
$O\setminus S$ is open, and the projection
\begin{equation*}
 P_{W^\perp}\colon O\setminus S\longrightarrow P_{W^\perp}(O)
\end{equation*}
is an open surjection. Its fibres are open disks with at most one
point removed, so they are nonempty and connected. Its image is
convex, hence connected. If $O\setminus S$ were the union of two
disjoint nonempty open sets, their projections would be disjoint
nonempty open sets covering $P_{W^\perp}(O)$: connectedness puts
each fibre entirely on one side, and openness of $P_{W^\perp}$
makes both images open. This proves connectedness for a ball.

For general $O$, join two points of $O\setminus S$ by a finite
chain of balls contained in $O$ in which consecutive balls
overlap. Each ball with $S$ removed is connected by the first
part, and consecutive ones have overlapping complements of $S$,
since $S$ contains no ball. Hence $O\setminus S$ is connected.
\end{proof}

\begin{lemma}[Polyhedral nonseparation]\label{lem:nonseparation}
Let $E$ be an $m$-dimensional Euclidean space, where $m\geq2$,
let $Y$ be a Hilbert space, and let $T\subset E$ be a compact
polyhedral set with nonempty connected interior
$O=\operatorname{int}_E T$ and $T=\cl O$.
Suppose that $w\colon T\to Y$ is continuous and piecewise
affine, and that its linear part on every full-dimensional piece
is $1$-Lipschitz and isometric on some two-dimensional subspace
of $E$. If $S\subset E$ is closed and
$\Lip(w,S\cap T)<1$, then $O\setminus S$ is connected
and dense in $T$.
\end{lemma}
\begin{proof}
Take a finite triangulation subordinate to the affine pieces.
Its full-dimensional simplices cover $T$, since $T=\cl O$.
On one such simplex $\Delta$, choose a two-dimensional subspace
$W$ on which the linear part is isometric. Two distinct points
of $S\cap\Delta$ cannot differ by a vector in $W$, because
their image distance would equal their source distance.
Lemma~\ref{lem:fibre-nonseparation} therefore shows that
\begin{equation*}
 A_\Delta=\operatorname{int}_E\Delta\setminus S
\end{equation*}
is connected and dense in $\Delta$.

The full-dimensional simplices can be joined through shared
facets. Indeed, a polygonal path in the connected open set $O$
can be perturbed slightly to avoid all faces of codimension at
least two and cross the remaining faces transversely. It then
passes through a finite sequence of full-dimensional simplices,
consecutive ones sharing a facet.

Consider such a facet $F$ between $\Delta$ and $\Delta'$.
For the isometric plane $W$ belonging to $\Delta$, the
intersection $W\cap\Span(F-F)$ contains a line.
Thus $\operatorname{int}_{\operatorname{Aff}F}F$ cannot lie in
$S$: it contains a nontrivial segment parallel to that line,
on which $w$ preserves distances. Choose a point $p$ of this
relative interior outside $S$. Since $F$ is shared, $p\in O$.
A small ball about $p$ therefore lies in $O\setminus S$ and
meets both $A_\Delta$ and $A_{\Delta'}$.
These balls join the sets $A_\Delta$ into a connected subset of
$O\setminus S$. This subset is dense in $O$, because each
$A_\Delta$ is dense in $\Delta$. Its relative closure is
therefore all of $O\setminus S$, which is connected and dense
in $T$.
\end{proof}

\subsection{One finite enlargement step}

We now assume that $X$ is separable and $\dim X\geq2$.
Brehm's theorem constructs isometric pieces only when the target
has enough dimensions. We therefore use an auxiliary Hilbert
target $G$ and a map $\pi\colon G\to Y$ as follows:
\begin{enumerate}
\item If $\dim Y\geq\dim X$, take $G=Y$ and $\pi=\Id_Y$.
\item If $2\leq\dim Y<\dim X$, enlarge $Y$ to a Hilbert space
      $G$ with $\dim G=\dim X$, and let $\pi=P_Y$.
\end{enumerate}
Since $X$ is separable, $Y$ is finite-dimensional in the second
case. One may take $G=Y\oplus\R^{\dim X-\dim Y}$ when $X$
is finite-dimensional and $G=Y\oplus\ell^2$ otherwise.
In the first case $Y$ need not be separable.

The enlargement is useful only if the eventual projection retains
two isometric directions. We ensure more: in the second case,
every full-dimensional affine piece maps its source space
isometrically onto a finite-dimensional subspace $G'\supset Y$.
After projection its linear part is a coisometry onto $Y$, and
is therefore isometric on a subspace of dimension $\dim Y\geq2$.

The next lemma is the induction step. A finite datum records all
values that must be retained, and a finite union of polytopes records
the region to be enclosed. We first choose finitely many new
points, then the finite-dimensional source and target, and only then
a polyhedral neighbourhood. This order avoids enlarging a polyhedron
after it has been constructed.

\begin{lemma}[Finite polyhedral enlargement]\label{lem:polyhedral-step}
Let $A\subset X$ be nonempty and closed, let
$u\colon A\to G$ be $1$-Lipschitz.
Let $P\subset X$ be nonempty and finite, and let
$a\colon P\to G$ be $1$-Lipschitz with
\begin{equation*}
 \|a(p)-u(y)\|\leq\|p-y\|-\delta
 \quad\text{for }p\in P,\ y\in A
\end{equation*}
for some $\delta>0$. Let $K\subset X\setminus D(u,A)$ be a finite
union of polytopes containing $P$.
There are a finite set $P'\supset P$, a $1$-Lipschitz datum
$a'\colon P'\to G$ extending $a$ with a positive uniform gap
against $u$, and a finite-dimensional linear subspace $E\subset X$
with $\dim E\geq2$ having the following property.
For each component $U$ of $X\setminus D(u,A)$ meeting $K$,
there is a compact polyhedral set $T_U\subset E\cap U$ such that
\begin{equation*}
 K\cap U\subset\operatorname{int}_E T_U,
 \qquad T_U=\cl\bigl(\operatorname{int}_E T_U\bigr).
\end{equation*}
Its interior is connected, and $P'\subset\bigcup_U T_U$.
The datum $a'$ fixes affine isometric
values on every simplex of a finite triangulation of $T_U$.
On each full-dimensional simplex, their composition with $\pi$
has an isometric linear part in the first target case and a
coisometric linear part onto $Y$ in the second.
\end{lemma}
\begin{proof}
The compact set $K$ meets only finitely many components $U$.
In each one, join the constituent polytopes of $K\cap U$ by
finitely many polygonal paths, obtaining a connected compact
polyhedral set $H_U\subset U$. Put $H=\bigcup_U H_U$.
By Lemma~\ref{lem:addition}, the combined datum on $A\cup P$
has a global $1$-Lipschitz extension $w$ and a constant $\gamma>0$
such that
\begin{equation*}
 \|w(x)-u(y)\|\leq\|x-y\|-\gamma
 \quad\text{for }x\in H,\ y\in A.
\end{equation*}
Choose $\eta>0$ with $4\eta<\gamma$ so that the
$\eta$-neighbourhood of each $H_U$ lies in $U$, and take a finite
$\eta$-net $Z\subset H$ containing $P$.

Choose finite-dimensional spaces $E\subset X$ and $G'\subset G$
containing $H$ and $w(Z)$, respectively, with
$\dim G'\geq\dim E\geq2$. In the projection case require
$Y\subset G'$ and $\dim G'=\dim E$.
If $X$ is finite-dimensional, take $E=X$ and, in the projection
case, $G'=G$. If $X$ is infinite-dimensional, first choose $G'$
containing $w(Z)$, and also $Y$ in the projection case, of dimension
at least $\max\{2,\dim\Span H\}$; then choose $E$
containing $H$ of that same dimension.

Choose an origin-centred full-dimensional closed cube $C\subset E$
whose points have norm less than $\eta$, and put $T_U=H_U+C$.
These Minkowski sums are finite unions of polytopes and lie
within distance $\eta$ of $H_U$. The connected set
$H_U+\operatorname{int}_E C$ contains $H_U$, lies in
$\operatorname{int}_E T_U$, and is dense in $T_U$.
Thus $T_U$ has connected interior and is its closure, as required.
Every point of $T=\bigcup_U T_U$ is within $2\eta$ of $Z$.

Brehm's theorem, including its unequal-dimension form, gives a
$1$-Lipschitz map $b\colon E\to G'$ interpolating $w|_Z$ and
affine isometric on the simplices of a finite triangulation of $T$.
For $x\in T$, choose $z\in Z$ with $\|x-z\|\leq2\eta$.
Then, for every $y\in A$,
\begin{equation*}
 \begin{aligned}
 \|b(x)-u(y)\|
 &\leq\|x-z\|+\|w(z)-u(y)\|\\
 &\leq\|x-z\|+\|z-y\|-\gamma
 \leq\|x-y\|-(\gamma-4\eta).
 \end{aligned}
\end{equation*}
Adjoin the triangulation vertices to $P$ and prescribe their
$b$-values; see Figure~\ref{fig:enlargement}. Since $P\subset Z$, this gives $P'$ and $a'$ with
unchanged old values and gap $\gamma-4\eta>0$.
Affine rigidity (Lemma~\ref{lem:leaves}) fixes $b$ on every
simplex from its vertex values.
In the projection case the full-dimensional linear parts map $E$
isometrically onto $G'\supset Y$, so composing with $\pi$ gives
coisometries onto $Y$.
\end{proof}

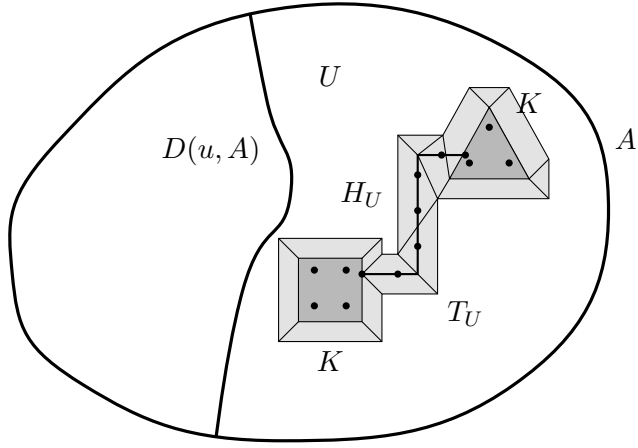
\begin{figure}[t]
\centering
\begin{tikzpicture}[scale=1.05]
 \def\blob{plot[smooth cycle, tension=0.8] coordinates
   {(0,0) (2.5,-0.7) (5.4,0) (6.6,2.2) (5.4,4.2) (2.2,4.7) (-0.3,3.1) (-0.9,1.4)}}
 \draw[very thick] \blob;
 \node[anchor=west] at (6.55,3.1) {$A$};
 \begin{scope}
  \clip \blob;
  \draw[very thick] plot[smooth, tension=0.9] coordinates
    {(1.6,-1.2) (2.0,1.2) (2.6,2.4) (2.3,3.6) (2.0,5.2)};
 \end{scope}
 \node at (1.6,2.95) {$D(u,A)$};
 \node at (3.1,3.9) {$U$};
 \fill[gray!22] (2.450,0.550) -- (3.750,0.550) -- (3.750,1.150) -- (4.450,1.150) -- (4.450,2.350) -- (5.850,2.350) -- (5.850,2.850) -- (5.350,3.750) -- (4.850,3.750) -- (4.517,3.150) -- (3.950,3.150) -- (3.950,1.650) -- (3.750,1.650) -- (3.750,1.850) -- (2.450,1.850) -- cycle;
 \draw[thin] (2.450,0.550) -- (3.750,0.550) -- (3.750,1.150) -- (4.450,1.150) -- (4.450,2.350) -- (5.850,2.350) -- (5.850,2.850) -- (5.350,3.750) -- (4.850,3.750) -- (4.517,3.150) -- (3.950,3.150) -- (3.950,1.650) -- (3.750,1.650) -- (3.750,1.850) -- (2.450,1.850) -- cycle;
 \draw[thin] (2.450,0.550) -- (2.700,0.800) (3.750,0.550) -- (3.500,0.800) (2.450,1.850) -- (2.700,1.600) (3.750,1.850) -- (3.500,1.600) (2.700,0.800) -- (3.500,1.600) (3.500,1.400) -- (3.750,1.150) (3.500,1.400) -- (3.750,1.650) (4.200,1.400) -- (4.450,1.150) (4.200,1.400) -- (3.950,1.650) (3.950,1.650) -- (4.450,2.350) (4.200,2.900) -- (3.950,3.150) (4.200,2.900) -- (4.450,2.350) (4.200,2.900) -- (4.517,3.150) (4.800,2.900) -- (4.600,2.600) (4.800,2.900) -- (5.100,3.500) (4.600,2.600) -- (4.450,2.350) (5.600,2.600) -- (5.850,2.350) (5.600,2.600) -- (5.850,2.850) (5.100,3.500) -- (4.850,3.750) (5.100,3.500) -- (5.350,3.750) (4.600,2.600) -- (4.517,3.150) (5.600,2.600) -- (5.100,3.500) (4.800,2.900) -- (5.600,2.600);
 \fill[gray!55] (2.700,0.800) -- (3.500,0.800) -- (3.500,1.600) -- (2.700,1.600) -- cycle; \draw[thin] (2.700,0.800) -- (3.500,0.800) -- (3.500,1.600) -- (2.700,1.600) -- cycle;
 \fill[gray!55] (4.600,2.600) -- (5.600,2.600) -- (5.100,3.500) -- cycle; \draw[thin] (4.600,2.600) -- (5.600,2.600) -- (5.100,3.500) -- cycle;
 \draw[thick] (3.500,1.400) -- (4.200,1.400) -- (4.200,2.900) -- (4.800,2.900);
 \foreach \p in {(3.500,1.400),(3.950,1.400),(4.200,1.750),(4.200,2.200),(4.200,2.650),(4.500,2.900),(4.800,2.900),(2.900,1.000),(3.300,1.000),(2.900,1.450),(3.300,1.450),(4.850,2.800),(5.350,2.800),(5.100,3.250)} \fill \p circle (0.045);
 \node at (3.100,0.300) {$K$}; \node at (5.600,3.550) {$K$};
 \node[anchor=east] at (3.900,2.400) {$H_U$};
 \node at (4.800,0.900) {$T_U$};
\end{tikzpicture}
\caption{One enlargement step (Lemma~\ref{lem:polyhedral-step})
inside a component $U$ of $X\setminus D(u,A)$. The polytopes of
$K\cap U$ (dark) are joined by a polygonal path into the connected
polyhedral set $H_U$. The dots form the finite $\eta$-net $Z$,
which contains $P$; only the values of the margin extension $w$ at
these points are kept. The polyhedron $T_U=H_U+C$ (light) is the
thickening of $H_U$ by a small cube; Brehm's theorem supplies
values that are affine isometric on the simplices of a finite
triangulation of $T_U$, and the vertices of this triangulation,
with these values, form the new datum $a'$.}
\label{fig:enlargement}
\end{figure}

\subsection{Iteration and approximation}

Lemma~\ref{lem:polyhedral-step} separates the two tasks in the
construction. It creates the required geometry while preserving
all finite data and a positive gap. We use a dense sequence to
make that geometry fill each component, and a separate finite set
of fixed values to obtain approximation on the prescribed compact
set. No estimate on the change from one whole extension to the
next is needed.

\begin{proposition}[Polyhedral approximation]\label{prop:skeleton}
Let $A\subset X$ be nonempty and closed, let
$u\colon A\to G$ be $1$-Lipschitz.
Given a global $1$-Lipschitz extension $v_0\colon X\to G$
of $u$, a nonempty compact set $Q\subset X$, and $\varepsilon>0$,
there is a global $1$-Lipschitz extension $F\colon X\to G$
of $u$ such that
\begin{equation*}
 \sup\{\|F(x)-v_0(x)\|\mid x\in Q\}<\varepsilon
\end{equation*}
and the following holds in every component $U$ of $X\setminus D(u,A)$.
There are increasing finite-dimensional linear subspaces
$E_{j,U}\subset X$ with $\dim E_{j,U}\geq2$ and compact
polyhedral sets $T_{j,U}\subset E_{j,U}\cap U$ such that
\begin{equation*}
 \begin{gathered}
 T_{j,U}=\cl\bigl(\operatorname{int}_{E_{j,U}}T_{j,U}\bigr),
 \text{ and }
 T_{j,U}\subset\operatorname{int}_{E_{j+1,U}}T_{j+1,U},\\
 \cl\bigl(\bigcup_j T_{j,U}\bigr)=\cl U.
 \end{gathered}
\end{equation*}
Each displayed interior is nonempty and connected. The restriction
of $F$ to each $T_{j,U}$ is piecewise affine and isometric on
every simplex of a finite triangulation. On each full-dimensional
simplex, the linear part of $\pi F$ is an isometric embedding
in the first case above and a coisometry onto $Y$ in the second.
\end{proposition}
\begin{proof}
\emph{Initialisation.}
If $D(u,A)=X$, take $F=v_0$. Otherwise, Lemma~\ref{lem:forced-degeneracy}
shows that every extension of $u$ already agrees with $v_0$ on
$D(u,A)$. Only values outside $D(u,A)$ need to be controlled.
Choose a countable dense set $\{q_1,q_2,\ldots\}$ in
$X\setminus D(u,A)$, choose $r>0$ with $4r<\varepsilon$, and take
a finite $r$-net $N\subset Q$. Put
\begin{equation*}
 P_0=(N\setminus D(u,A))\cup\{q_1\},\qquad K_0=P_0.
\end{equation*}
Lemma~\ref{lem:addition}, with no additional finite datum, gives
a global extension $s$ of $u$ with a positive uniform gap
against $u$ on $P_0$. For sufficiently small $t\in(0,1)$, the
finite datum
\begin{equation*}
 a_0=((1-t)v_0+ts)|_{P_0}
\end{equation*}
has the same property and differs from $v_0$ by less than
$\varepsilon/2$ on $P_0$.

\emph{The induction.}
At stage $j$, retain a finite $1$-Lipschitz datum
$a_j\colon P_j\to G$, extending every earlier finite datum,
with a positive uniform gap against $u$ on $A$. Also retain a
finite union of prescribed simplices $K_j\subset X\setminus D(u,A)$
containing $P_j$, with all simplex vertices in $P_j$.
Their affine isometric values are fixed by these vertices.
Initially the simplices are just the points of $P_0$.

Apply Lemma~\ref{lem:polyhedral-step} to the original datum
$u\colon A\to G$, the finite datum $a_j$, and the set
$K_j\cup\{q_{j+2}\}$. Let $K_{j+1}$ be the union of the
polyhedra it constructs. The new vertex datum
$a_{j+1}\colon P_{j+1}\to G$ has positive gap and preserves
all old values. Affine rigidity (Lemma~\ref{lem:leaves}) therefore
preserves every old
simplex, and the new triangulations give the next stage's
prescribed simplices. The old polyhedra lie in the relative
interiors of the new ones.

Every component of $X\setminus D(u,A)$ appears once a point $q_j$
in it is added, and appears at every later stage. Its nested
polyhedra contain all the $q_j$ in that component, so their union
is dense there. Their source subspaces increase because a
polyhedron with nonempty interior in a linear subspace spans
that subspace. Reindexing in each component gives the asserted
$E_{j,U}$ and $T_{j,U}$.

\emph{Extension and approximation.}
Kirszbraun's theorem extends the coherent $1$-Lipschitz datum on
$A\cup\bigcup_jP_j$ to a global map $F$. Affine rigidity
preserves every prescribed simplex, so the asserted isometric
and coisometric structure follows from the finite enlargement
lemma. No convergence of the auxiliary maps is required.
The map $F$ differs from $v_0$ by less than $\varepsilon/2$ on
$N\setminus D(u,A)$ and agrees with $v_0$ on $N\cap D(u,A)$ by
Lemma~\ref{lem:forced-degeneracy}. The two $1$-Lipschitz
bounds therefore give
\begin{equation*}
 \begin{aligned}
 \sup\{\|F(x)-v_0(x)\|\mid x\in Q\}
 &\leq2r+\max\{\|F(z)-v_0(z)\|\mid z\in N\}\\
 &<2r+\varepsilon/2<\varepsilon.
 \end{aligned}
\end{equation*}
\end{proof}

\subsection{Boundary saturation}\label{subsec:saturation}

We now show that the extension $\widetilde u$ constructed in
Proposition~\ref{prop:skeleton} is saturated: there is no nonempty
open $V\subset\Omega$ with $\Lip(\widetilde u,\partial V)<1$.
Suppose there were one. Lemma~\ref{lem:paste} gives an
extension $w$ of $u$ that is strictly contracting on $V$ and shows
that $V$ avoids $D(u,A)$. Let $U$ be a component of
$\Omega\setminus D(u,A)$ meeting $V$. By
Lemma~\ref{lem:nonseparation}, the interior of each polyhedron
$T_{j,U}$ with $\partial V$ removed is connected and dense in
$T_{j,U}$. Since the polyhedra are nested with dense union, from
some index on each of these connected sets meets $V$, hence lies
in $V$, and therefore $U\subset\cl V$. Thus $w$ is strictly
contracting on $U$, and Lemma~\ref{lem:contact} forces $U$ to be
all of $\Omega$, contradicting the Lipschitz constant of the
boundary datum. Figure~\ref{fig:saturation} illustrates these two
propagation steps.

\begin{figure}[t]
\centering
\begin{tikzpicture}[scale=1.05]
 \def\blob{plot[smooth cycle, tension=0.8] coordinates
   {(0,0) (2.5,-0.7) (5.4,0) (6.6,2.2) (5.4,4.2) (2.2,4.7) (-0.3,3.1) (-0.9,1.4)}}
 \draw[very thick] \blob;
 \node[anchor=west] at (6.55,3.1) {$A=\partial\Omega$};
 \node at (0.55,1.1) {$\Omega$};
 \begin{scope}
  \clip \blob;
  \draw[very thick] plot[smooth, tension=0.9] coordinates
    {(1.6,-1.2) (2.0,1.2) (2.6,2.4) (2.3,3.6) (2.0,5.2)};
 \end{scope}
 \node at (1.75,2.95) {$D(u,A)$};
 \fill[gray!22] (3.8,0.7) -- (5.7,0.9) -- (5.9,2.6) -- (4.8,3.6) -- (3.85,3.0) -- cycle;
 \draw[thin] (3.8,0.7) -- (5.7,0.9) -- (5.9,2.6) -- (4.8,3.6) -- (3.85,3.0) -- cycle;
 \fill[gray!50] (4.25,1.3) -- (5.15,1.45) -- (5.25,2.35) -- (4.6,2.95) -- (4.15,2.4) -- cycle;
 \draw[thin] (4.25,1.3) -- (5.15,1.45) -- (5.25,2.35) -- (4.6,2.95) -- (4.15,2.4) -- cycle;
 \node at (4.7,2.1) {$T_1$};
 \node at (4.05,1.0) {$T_2$};
 \node at (3.35,3.85) {$U$};
 \draw[dashed, thick] plot[smooth cycle, tension=0.9] coordinates
   {(3.45,0.45) (4.7,0.25) (6.1,0.7) (6.25,2.3) (5.5,3.8) (4.35,4.0) (3.4,3.3) (3.3,1.8)};
 \node at (4.15,0.18) {$\partial V$};
\end{tikzpicture}
\caption{Propagation inside a component $U$ of $\Omega\setminus D(u,A)$.
The nested polyhedra $T_1\subset T_2$ are shaded; the comparison
boundary $\partial V$ is dashed. If its Lipschitz constant is
less than one, $\partial V$ cannot disconnect the interior of any
polyhedron (Lemma~\ref{lem:nonseparation}), so every sufficiently
late polyhedron lies in $\cl V$. Their dense union then gives
$U\subset\cl V$. Affine isometries are prescribed on the
simplices of each polyhedron, not on the polyhedron as a whole.}
\label{fig:saturation}
\end{figure}
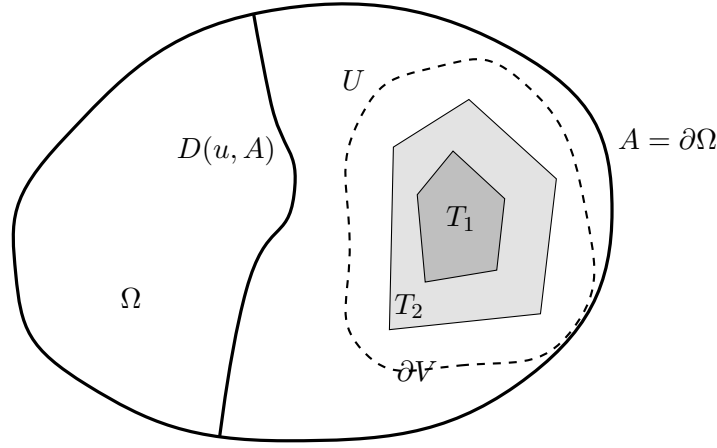

\begin{proof}[Proof of Theorem~\ref{thm:domain} for separable sources]
If $L=0$, the extension is constant. If $\dim X=1$ and $L>0$,
then $\Omega$ is a bounded interval: a ray has only one boundary
point and would give $L=0$. Equality in the triangle inequality
forces every optimal extension on the interval to be the affine
interpolation of its endpoint values. This map is saturated.
We may therefore assume $\dim X\geq2$ and, after dividing the
data and requested approximation error by $L$, that $L=1$.

Put $A=\partial\Omega$, and extend $v_0$ globally by
Kirszbraun's theorem. Choose $G$ and $\pi$ as above, regard
$u$ and $v_0$ as $G$-valued.
Proposition~\ref{prop:skeleton} gives $F$; put
$\widetilde u=\pi F$. Projection fixes $v_0$ and does not
increase distances, so $\widetilde u$ has the required
approximation and global $1$-Lipschitz bound.

Suppose, for a contradiction, that some nonempty open
$V\subset\Omega$ has $k=\Lip(\widetilde u,\partial V)<1$.
Apply Lemma~\ref{lem:paste} to $\widetilde u$ as a $G$-valued
map. It gives a global extension $w$ of $u$ that is
$k$-Lipschitz on $V$, and shows that $V\cap D(u,A)=\emptyset$.

\emph{Propagation through one component.}
Choose a component $U$ of $\Omega\setminus D(u,A)$ meeting $V$.
Since $\partial\Omega\subset D(u,A)$, it is also a component of
$X\setminus D(u,A)$, and the polyhedra of
Proposition~\ref{prop:skeleton} are available in $U$.
On every full-dimensional simplex, $\widetilde u$ has a linear
part isometric in at least two directions. Apply
Lemma~\ref{lem:nonseparation} in $E_{j,U}$ to the closed set
$E_{j,U}\cap\partial V$. It shows that
\begin{equation*}
 O_j=\operatorname{int}_{E_{j,U}}T_{j,U}\setminus\partial V
\end{equation*}
is connected and dense in $T_{j,U}$.

The union of the polyhedra is dense in $U$, so some $O_{j_0}$
meets the nonempty open set $V\cap U$.
The nesting $T_{j,U}\subset\operatorname{int}_{E_{j+1,U}}T_{j+1,U}$
gives $O_j\subset O_{j+1}$. Hence every $O_j$ with $j\geq j_0$
meets $V$. Such an $O_j$ cannot have points on both sides of
$\partial V$: it is connected and avoids that boundary.
It follows that $O_j\subset V$, and density gives
$T_{j,U}\subset\cl V$. Taking the union of these later
polyhedra yields $U\subset\cl V$. Continuity now extends the
$k$-Lipschitz estimate for $w$ from $V$ to all of $U$.

\emph{Propagation to the original boundary.}
Lemma~\ref{lem:contact}, applied to $u$ and $w$, gives
$\partial U\cap\Omega=\emptyset$. Thus $U$ is relatively
closed as well as relatively open in $\Omega$. Since $\Omega$
is connected, $U=\Omega$. The map $w$ is therefore
$k$-Lipschitz on $\Omega$ and, by continuity, on its closure.
It agrees with $u$ on $A$, contradicting
$\Lip(u,A)=1>k$.

We have proved $\Lip(\widetilde u,\partial V)=1$ for every
nonempty open $V\subset\Omega$. Since $\partial V\subset\cl V$,
continuity and the global bound give
$\Lip(\widetilde u,V)=1$ as well. Restricting to $\cl\Omega$
and restoring the scale $L$ proves the theorem.
\end{proof}

\begin{remark}[Structure of the constructed extensions]\label{rem:structure}
Let $X$ be separable, normalise $L=1$, and let $\widetilde u=\pi F$
be the extension constructed above. In each component $U$ of
$\Omega\setminus D(u,\partial\Omega)$, the map $F$ is affine
isometric on every simplex of the nested polyhedra $T_{j,U}$ of
Proposition~\ref{prop:skeleton}, whose union is dense in $U$. By
Lemma~\ref{lem:leaves}, each simplex lies in a leaf of $F$; the
polyhedron need not, since adjacent simplices may carry different
isometries. As the full-dimensional simplices cover $T_{j,U}$ and
$\dim E_{j,U}\geq2$, the leaves of $F$ of dimension at least two
have dense union in $\Omega\setminus D(u,\partial\Omega)$. The same
holds for $\widetilde u$: either $\widetilde u=F$, or
$\dim Y<\dim X$ and the linear part $R$ of $\widetilde u$ on a
full-dimensional simplex $\Delta$ is a coisometry onto $Y$, so that
$\widetilde u$ preserves distances on the slices of $\Delta$
parallel to $(\ker R)^\perp$, which have dimension $\dim Y\geq2$
and fill the interior of $\Delta$.

Off the polyhedra nothing is prescribed, and in an
infinite-dimensional source they have empty interior. When $X$ is
finite-dimensional, the union of their interiors is a dense open
subset of $\Omega\setminus D(u,\partial\Omega)$ on which
$\widetilde u$ is locally piecewise affine. Whether this set can be
chosen of full measure, and whether every saturated extension then
has derivative of norm one almost everywhere, remain open here.
\end{remark}

\section{Passage to arbitrary Hilbert spaces}\label{sec:nonseparable}

A map constructed on a separable section may violate boundary
constraints from outside that section. We first choose sections
that retain all these constraints, then enlarge them while keeping
previously assigned values fixed. The Kuratowski--Zorn lemma
assembles a coherent family of these extensions. Exact agreement
preserves saturation when we pass to the closure of its union.

Throughout this section, $X$ is nonseparable,
$\Omega\subsetneq X$ is nonempty, connected and open,
$A=\partial\Omega$, and $u\colon A\to Y$ satisfies
$\Lip(u,A)=1$, with $\dim Y\geq2$. Fix $x_*\in\Omega$.
To ensure that restriction never lowers the boundary Lipschitz
constant, choose distinct points $a_n,b_n\in A$ such that
\begin{equation*}
 \frac{\|u(a_n)-u(b_n)\|}{\|a_n-b_n\|}\longrightarrow1.
\end{equation*}
Choose a separable closed linear subspace $E_*$ containing
$x_*$ and sequences in $\Omega$ converging to every $a_n,b_n$.
For a closed linear subspace $E$, write
$\Omega_E=\Omega\cap E$. Subscripts on boundaries and closures
indicate the ambient subspace. Whenever $E\supset E_*$, the
retained sequences ensure that $a_n,b_n\in\partial_E\Omega_E$.
Since always
$\partial_E\Omega_E\subset A\cap E$, we have
\begin{equation*}
 \Lip(u,\partial_E\Omega_E)=1.
\end{equation*}
Thus we only need to preserve boundary accessibility for the
countably many pairs witnessing the Lipschitz constant.

\subsection{Sections retaining the boundary data}

We enlarge the source and target together. Let $\Phi=\Phi_{u,A}$
be the margin function \eqref{eq:margin-function}. A pair $(E,J)$
of closed linear subspaces of $X$ and $Y$ is \emph{compatible} if
$E_*\subset E$, $\dim J\geq2$, $u(A\cap E)\subset J$,
$\Omega_E$ is connected, and
\begin{equation}\label{eq:reflection}
 \Phi(x,p)=\sup\{\|p-u(y)\|-\|x-y\|\mid y\in A\cap E\}
 \qquad(x\in E,\ p\in J).
\end{equation}
This identity says that, for values in $J$, the boundary points
in $E$ impose all the constraints of the original datum.
Thus every $1$-Lipschitz map
$v\colon E\to J$ extending $u|_{A\cap E}$ satisfies
\begin{equation}\label{eq:global-compatibility}
 \|v(x)-u(y)\|\leq\|x-y\|
 \qquad(x\in E,\ y\in A),
\end{equation}
because \eqref{eq:reflection} gives $\Phi(x,v(x))\leq0$.
We may therefore keep $v$ unchanged when further original boundary
values are introduced.

Some points of $A\cap E$ may lie outside
$\cl_E\Omega_E$. We keep their prescribed values by extending
a map on $\Omega_E$ through pasting into an existing global
extension on $E$.

\begin{lemma}[Compatible subspaces of controlled density]\label{lem:compatible}
Let $\mu$ be an infinite cardinal, and let $S\subset X$ and
$T\subset Y$ have density at most $\mu$. There is a compatible
pair $(E,J)$ with $S\subset E$, $T\subset J$, and
\begin{equation*}
 \operatorname{dens}E\leq\mu,
 \qquad \operatorname{dens}J\leq\mu.
\end{equation*}
\end{lemma}
\begin{proof}
Fix a global $1$-Lipschitz extension $h\colon X\to Y$ of $u$
and a two-dimensional subspace $W\subset Y$. We need only
construct the source sections: their target sections will contain
the corresponding values of $h$.

Let $E_0$ be the closed linear span of $S\cup E_*$.
Given $E_n$, put
\begin{equation*}
 J_n=\cl\bigl(\Span(T\cup W\cup h(E_n))\bigr).
\end{equation*}
Choose dense subsets of $E_n$ and $J_n$ of cardinality at most
$\mu$. For every pair $(x,p)$ in these subsets and every positive
integer $r$, choose $y\in A$ with
\begin{equation*}
 \|p-u(y)\|-\|x-y\|>\Phi(x,p)-1/r.
\end{equation*}
Choose also a dense subset of $\Omega\cap E_n$ of cardinality
at most $\mu$ and join its points to $x_*$ by finite polygonal
paths in $\Omega$. Let $E_{n+1}$ be the closed linear span of
$E_n$, all the chosen points $y$, and all the path vertices.
Each step preserves the density bound $\mu$, because $h$ is
Lipschitz and at most $\mu$ vectors are added to the source.

Set
\begin{equation*}
 E=\cl\bigl(\bigcup_nE_n\bigr),
 \qquad J=\cl\bigl(\bigcup_nJ_n\bigr).
\end{equation*}
Both spaces have density at most $\mu$. Continuity gives
$h(E)\subset J$, so in particular $u(A\cap E)\subset J$.
Every boundary point selected above lies in $A\cap E$.
The margin choices therefore give \eqref{eq:reflection} on the
chosen dense pairs, and then on $E\times J$ by continuity.
The chosen paths have a connected union containing $x_*$, dense
in $\Omega_E$: approximate a point of this open section first
from some $E_n$ and then from the chosen dense subset of
$\Omega\cap E_n$. Thus $\Omega_E$ is connected.
\end{proof}

\subsection{Enlarging a section without changing its values}

Once a saturated map on $E$ satisfies
\eqref{eq:global-compatibility}, we can keep it fixed and fill
$\Omega_F\setminus E$ in a larger compatible section $F$. If $E$
has codimension at least two in $F$, this set is connected, and
its boundary contains $\Omega_E$, where the old map already has
Lipschitz constant one. Only the new section needs to be
compatible: from the old map we use its saturation and its
inequalities against the original boundary values.

\begin{lemma}[Extension preserving a saturated subspace]\label{lem:saturated-subspace}
Let $(F,J)$ be compatible, and let $E\subset F$ be a closed
linear subspace containing $x_*$ such that $F\cap E^\perp$ has
dimension at least two. Assume that
Theorem~\ref{thm:domain} holds for source $F$ and every Hilbert
target of dimension at least two. Suppose that
$\widetilde u_E\colon E\to J$ is $1$-Lipschitz, satisfies
\eqref{eq:global-compatibility}, and is saturated with constant
one on $\Omega_E$. Then it has a $1$-Lipschitz extension
$\widetilde u_F\colon F\to J$ with the same two properties.
\end{lemma}
\begin{proof}
Prescribe the datum $a$ to equal $\widetilde u_E$ on $E$ and $u$
on $A\cap F$. Inequality~\eqref{eq:global-compatibility} shows
that $a$ is $1$-Lipschitz on the closed set $B=E\cup(A\cap F)$.
Extend it by Kirszbraun to $h\colon F\to J$.

Let $W\subset F\cap E^\perp$ be two-dimensional. Every affine
plane parallel to $W$ meets $E$ in at most one point, and
$E\cap\Omega_F$ is relatively closed in the connected open set
$\Omega_F$. Lemma~\ref{lem:fibre-nonseparation} therefore shows
that $U=\Omega_F\setminus E$ is connected and dense in $\Omega_F$.
Its boundary in $F$ lies in $B$ and contains $\Omega_E$, since
every point of $\Omega_E$ is a limit of points of $U$. Saturation
of the old map gives $\Lip(\widetilde u_E,\Omega_E)=1$, and hence
$\Lip(a,\partial_F U)=1$. Apply the assumed domain theorem on
$U$, with boundary datum $a|_{\partial_F U}$ and initial extension
$h$, and replace $h|_U$ by the resulting saturated extension.
Lemma~\ref{lem:pasting} gives a globally $1$-Lipschitz map
$\widetilde u_F\colon F\to J$ extending $a$.

To check saturation on all of $\Omega_F$, take a nonempty open
set $V\subset\Omega_F$. If $V$ meets $E$, the old saturation
already supplies the required boundary estimate, since
$\partial_E(V\cap E)\subset\partial_F V$:
\begin{equation*}
 1=\Lip(\widetilde u_E,\partial_E(V\cap E))
   \leq\Lip(\widetilde u_F,\partial_F V).
\end{equation*}
If $V$ avoids $E$, then $V\subset U$, and saturation on $U$
gives the same estimate. The global bound and continuity
then give $\Lip(\widetilde u_F,V)=
\Lip(\widetilde u_F,\partial_F V)=1$. Compatibility of
$(F,J)$ gives \eqref{eq:global-compatibility} for the new map.
\end{proof}

Saturation need not survive limits of maps whose old values keep
changing. Here the maps agree exactly on every earlier section.
Every open test set in the closure meets one of these sections,
and the boundary Lipschitz estimates obtained there remain valid
for the final map. This is the reason saturation survives the next limit.

\begin{lemma}[Completion of an increasing chain]\label{lem:coherent-limit}
Let $(E_\alpha)_{\alpha\in I}$ be a nonempty increasing chain
of closed linear subspaces. Suppose the $1$-Lipschitz maps
$\widetilde u_\alpha\colon E_\alpha\to Y$ agree on their
common domains, satisfy \eqref{eq:global-compatibility}, and
are saturated with constant one on $\Omega\cap E_\alpha$.
Their union extends uniquely to a $1$-Lipschitz map on
$E=\cl(\bigcup_\alpha E_\alpha)$ with the same properties.
In particular it agrees with $u$ on $A\cap E$.
\end{lemma}
\begin{proof}
Exact agreement makes the union map well-defined and
$1$-Lipschitz; denote its unique extension to $E$ by
$\widetilde u$. Inequality~\eqref{eq:global-compatibility} passes to this
extension. At $x=y\in A\cap E$ it gives
$\widetilde u(x)=u(x)$.

Let $V\subset\Omega_E$ be nonempty and relatively open.
It meets some $E_\alpha$, and the boundary of its section is
contained in its full boundary. Because the map is unchanged on
$E_\alpha$,
\begin{equation*}
 1=\Lip(\widetilde u_\alpha,
                  \partial_{E_\alpha}(V\cap E_\alpha))
 \leq\Lip(\widetilde u,\partial_E V)\leq1.
\end{equation*}
Continuity gives $\Lip(\widetilde u,V)=1$ as well.
\end{proof}

\subsection{The maximality argument}

\begin{proof}[Completion of Theorem~\ref{thm:domain}]
Suppose the theorem fails, and choose a counterexample with least
source density $\kappa$, allowing all Hilbert targets of dimension
at least two. Then $\kappa>\aleph_0$, and the theorem holds for
every source of smaller density and every such target. Normalise
to $L=1$, the case $L=0$ being immediate, and extend $v_0$ globally
by Kirszbraun. Index an orthonormal basis of $X$ by
$(e_\alpha)_{\alpha<\kappa}$, viewing $\kappa$ as an initial ordinal.

We first arrange the approximation on a separable section.
Lemma~\ref{lem:compatible}, applied to $Q\cup\{e_0\}$ and
$v_0(Q)$, gives a compatible separable pair $(E_0,J_0)$.
The map $h=P_{J_0}v_0$ agrees with $u$ on $A\cap E_0$ and
with $v_0$ on $Q$. Apply the separable theorem on $\Omega_{E_0}$
to the boundary datum $u|_{\partial_{E_0}\Omega_{E_0}}$, the
initial extension $h$, and the nonempty compact set
\begin{equation*}
 K=(Q\cap\cl_{E_0}\Omega_{E_0})\cup\{x_*\}.
\end{equation*}
Paste the resulting map into $h$ on $E_0$ using
Lemma~\ref{lem:pasting}. The resulting $1$-Lipschitz map
$\widetilde u_0\colon E_0\to J_0$ agrees with $u$ on all
of $A\cap E_0$, is saturated on $\Omega_{E_0}$, and is
within $\varepsilon$ of $v_0$ on $Q$: on the remaining points
of $Q$, it still equals $h=v_0$. Reflection gives
\eqref{eq:global-compatibility} for $\widetilde u_0$.

We use the Kuratowski--Zorn lemma to assemble a coherent family
of such maps. A \emph{partial construction} of length
$1\leq\gamma\leq\kappa$ consists of increasing closed linear
subspaces $E_\alpha$ and $1$-Lipschitz maps
$\widetilde u_\alpha\colon E_\alpha\to Y$, $\alpha<\gamma$,
starting with $(E_0,\widetilde u_0)$, that extend one another,
satisfy \eqref{eq:global-compatibility}, and are saturated with
constant one on $\Omega_{E_\alpha}$. Require also
\begin{equation*}
 e_\alpha\in E_\alpha,
 \qquad
 \operatorname{dens}E_\alpha
 \leq\mu_\alpha:=\max\{\aleph_0,|\alpha|\}<\kappa.
\end{equation*}
Order partial constructions by extension of the indexed family:
a longer one must agree with the shorter one at every earlier
index. The one-term construction exists, and the union of a
nonempty chain is again a partial construction, of length at
most $\kappa$. Thus the Kuratowski--Zorn lemma gives a maximal
partial construction, of some length $\gamma$.

If $\gamma<\kappa$, Lemma~\ref{lem:coherent-limit} gives a
map $v\colon E\to Y$ on the closed source union.
There are at most $|\gamma|$ earlier sections, each of density
at most $\mu_\gamma$, so
\begin{equation*}
 \operatorname{dens}v(E)\leq\operatorname{dens}E
 \leq\mu_\gamma<\kappa.
\end{equation*}
This estimate holds for every $\gamma<\kappa$, even when
$\kappa$ is singular. Since $\operatorname{dens}E<\kappa$, the
orthogonal complement $E^\perp$ has dimension at least two, for
otherwise $X$ would have density at most $\operatorname{dens}E$;
choose linearly independent $r,s\in E^\perp$.
Apply Lemma~\ref{lem:compatible} to
$E\cup\{e_\gamma,r,s\}$ and $v(E)$, with bound $\mu_\gamma$,
to obtain $(F,J)$. Then $F\cap E^\perp$ has dimension at least
two, and $F$ has density less than $\kappa$,
so Theorem~\ref{thm:domain} and Lemma~\ref{lem:saturated-subspace}
extend $v$ to a saturated map $w\colon F\to J$ satisfying
\eqref{eq:global-compatibility}. Appending $(F,w)$ at index $\gamma$
strictly extends the partial construction, a contradiction.
This argument treats successor and limit lengths in the same way.

Hence $\gamma=\kappa$. The source union contains every basis
vector, so its closure is $X$. Lemma~\ref{lem:coherent-limit}
now gives a saturated $1$-Lipschitz extension on $X$ that agrees
with $u$ and retains the approximation on $Q$. Its restriction
to $\cl\Omega$ contradicts the choice of counterexample.
\end{proof}

\section{Arbitrary prescribed sets and all dimensions}\label{sec:existence}

For a general prescribed set $A$, absolute minimality requires
optimality on each component of $X\setminus A$, with the
Lipschitz constant determined by that component's boundary.
The global constant $\Lip(u,A)$ can be larger. We therefore
solve the boundary problems separately at their own constants;
Lemma~\ref{lem:pasting} then gives comparison even on open sets that
meet several components.

\begin{proof}[Proof of Theorem~\ref{thm:closed}]
Put $L=\Lip(u,A)$. The cases $A=\emptyset$, $A=X$, or
$L=0$ are immediate.

If $\dim Y=1$, identify $Y$ with $\R$ and apply Milman's
existence theorem \cite[Theorems~4.2 and~5.4, Section~6]{Milman}
to $X\setminus A$, with boundary values
$u|_{\partial(X\setminus A)}$. Its hypotheses hold because
the Hilbert metric is a length metric and Lipschitz data are
linearly bounded. The resulting extension has Lipschitz constant
at most $L$ and satisfies comparison on every open subset of
$X\setminus A$. Retaining $u$ on $A$ gives the required global
extension by Lemma~\ref{lem:pasting}\textup{(i)}.
No boundedness or separability assumption is needed.

Suppose now that $\dim Y\geq2$. On every component $\Omega_i$
of $X\setminus A$, choose a saturated extension of
$u|_{\partial\Omega_i}$ with constant
$L_i=\Lip(u,\partial\Omega_i)$ by Theorem~\ref{thm:domain};
when $L_i=0$, use the constant extension. Since $L_i\leq L$,
Lemma~\ref{lem:pasting}\textup{(i)} and the observation following
it paste these maps with $u$ into a global $L$-Lipschitz
extension $\widetilde u$.

For a nonempty open $V\subset X\setminus A$, put
$k=\Lip(\widetilde u,\partial V)$ and $V_i=V\cap\Omega_i$.
Since $\partial\Omega_i\subset A$, we have
$\partial V_i\subset\partial V$. Comparison on each component
therefore gives $\Lip(\widetilde u,V_i)\leq k$ whenever
$V_i\ne\emptyset$. Thus $\widetilde u$ is locally
$k$-Lipschitz on $V$, and Lemma~\ref{lem:pasting}\textup{(ii)}
gives $\Lip(\widetilde u,V)\leq k$. The reverse inequality
follows by continuity.
\end{proof}

The component bounds are also the only obstruction to approximating
an optimal global extension by AMLEs.

\begin{corollary}[Closure of the class of AMLEs]\label{cor:density}
Let $X,Y$ be Hilbert spaces with $\dim Y\geq2$, let $A\subset X$
be closed, and let $u\colon A\to Y$ be Lipschitz. Among optimal
extensions of $u$, the closure of the absolutely minimising
extensions for uniform convergence on compact sets consists
precisely of the maps $v$ satisfying
\begin{equation*}
 \Lip(v,\Omega_i)=\Lip(u,\partial\Omega_i)
 \text{ for every component }\Omega_i\text{ of }X\setminus A.
\end{equation*}
\end{corollary}
\begin{proof}
The cases $A=\emptyset$ and $A=X$ are immediate. Every AMLE
satisfies the displayed equalities by comparison on $\Omega_i$.
For any extension of $u$, continuity gives
$\Lip(v,\Omega_i)\geq\Lip(u,\partial\Omega_i)$, so each equality
is equivalent to a collection of closed pairwise Lipschitz
inequalities. These inequalities persist under uniform convergence
on compact sets.

Conversely, suppose $v$ satisfies them, and fix a nonempty compact
$Q\subset X$ and $\varepsilon>0$. In each component whose closure
meets $Q$, apply Theorem~\ref{thm:domain} to $v|_{\cl\Omega_i}$
with compact set $Q\cap\cl\Omega_i$ and error $\varepsilon/2$;
in each remaining component choose any saturated extension.
The preceding proof pastes these maps with $u$ into an AMLE
$\widetilde u$, and
\begin{equation*}
 \sup\{\|\widetilde u(x)-v(x)\|\mid x\in Q\}
 \leq\varepsilon/2<\varepsilon.\qedhere
\end{equation*}
\end{proof}

For example, let $X=\R$, choose a unit vector $e\in Y$, and
prescribe $u(0)=u(1)=0$ and $u(2)=e$ on $A=\{0,1,2\}$.
We have $\Lip(u,A)=1$, and every AMLE vanishes
on $[0,1]$, where the component boundary constant is zero.
Prescribing in addition $v(1/2)=e/2$ still gives $1$-Lipschitz
data and hence, by Kirszbraun's theorem, an optimal global
extension. This optimal extension cannot be approximated by AMLEs
at $1/2$.
In fact, the AMLE is unique: it is affine on $[1,2]$ and constant
on each side of that interval. Global optimality alone therefore
does not capture the component bounds in Corollary~\ref{cor:density}.

\section{Uniqueness, the scalar case, and uniform convergence}\label{sec:consequences}

On a connected domain, density has an immediate consequence for
uniqueness: an AMLE can be unique only if the optimal extension
is already unique. Kirszbraun's theorem makes the latter condition
pointwise. It asks whether the boundary inequalities force a single
value at every interior point.

\begin{corollary}[Uniqueness]\label{cor:uniqueness}
Let $X,Y$ be Hilbert spaces with $\dim Y\geq2$, let
$\Omega\subsetneq X$ be nonempty, connected and open, and let
$u\colon\partial\Omega\to Y$ be Lipschitz. Put
$L=\Lip(u,\partial\Omega)$. The following conditions are equivalent:
\begin{enumerate}
\item $u$ has a unique absolutely minimising extension to
      $\cl\Omega$;
\item $u$ has a unique $L$-Lipschitz extension to $\cl\Omega$;
\item for every $x\in\Omega$, the admissible-value set
\begin{equation*}
 \mathcal P_L(u,\partial\Omega,x)
 =\{p\in Y\mid \|p-u(a)\|\leq L\|x-a\|
       \text{ for every }a\in\partial\Omega\}
\end{equation*}
consists of exactly one point.
\end{enumerate}
If these conditions fail, $u$ has infinitely many absolutely
minimising extensions.
\end{corollary}
\begin{proof}
Let $\mathcal E$ be the convex set of $L$-Lipschitz extensions
of $u$, with the topology of uniform convergence on compact sets.
The AMLEs form a nonempty dense subset by
Theorem~\ref{thm:domain}. If there were finitely many, they
would be closed and hence equal $\mathcal E$; convexity would
then force $\mathcal E$ to be a singleton. Conversely, existence
makes the unique member of $\mathcal E$ an AMLE. This proves
the equivalence of the first two conditions and the last assertion.

For each $x\in\Omega$, Kirszbraun's theorem identifies
$\mathcal P_L(u,\partial\Omega,x)$ with the values $v(x)$
attained by $v\in\mathcal E$. Thus the second and third
conditions are equivalent.
\end{proof}

For an arbitrary closed prescribed set, the criterion applies to
each component of its complement. More precisely, when
$\dim Y\geq2$, an AMLE is unique if and only if every component
boundary datum has a unique optimal extension. In the
admissible-value criterion, one uses
$L_i=\Lip(u,\partial\Omega_i)$ on $\Omega_i$.
These choices are independent: Lemma~\ref{lem:pasting} pastes
any collection of componentwise optimal extensions with the
original datum into a global optimal extension.

\begin{remark}[The contrast with bounded scalar data]\label{rem:dichotomy}
Bounded real-valued Lipschitz data on a nonempty closed subset of
a Hilbert space have a unique AMLE in the sense of
\eqref{eq:amle}, by \cite[Theorem~1.4]{PSSW}. On bounded
Euclidean domains, uniqueness is Jensen's theorem \cite{Jensen}.
Thus, for bounded scalar data, absolute minimality selects one
extension even when optimal extensions are not unique. For targets
of dimension at least two, Corollary~\ref{cor:uniqueness} gives
the different conclusion: any freedom among optimal extensions
on a connected domain entails infinitely many AMLEs, dense in
the optimal class. We make no corresponding uniqueness assertion
for arbitrary unbounded scalar data. The cited scalar theorem
assumes boundedness, and \cite[Section~5.2]{PSSW} gives
nonuniqueness examples on general length spaces.

Even a smooth scalar boundary datum can have no saturated
extension. Identify $\R^2$ with $\mathbb C$, put $\Omega=B_1(2)$,
and let $f(z)=\arg z$ on $\cl\Omega$. This is an AMLE of its
boundary datum. Indeed, for an open $V\subset\Omega$ and $z\in V$,
the component of $V\cap\{|w|=|z|\}$ through $z$ is a circular arc
with distinct endpoints $p,q\in\partial V$. Its length is
$|z|\,|f(p)-f(q)|$ and is at least $|p-q|$. Hence
\begin{equation*}
 \Lip(f,\partial V)\geq\frac1{|z|}=\|\nabla f(z)\|,
\end{equation*}
and Lemma~\ref{lem:pasting}\textup{(ii)} gives comparison on $V$.
The Lipschitz constant on $\cl\Omega$ is one, detected by the
tangential derivative of the boundary datum at $z=1$, whereas
\begin{equation*}
 \Lip(f,B_{1/4}(2))\leq4/7<1.
\end{equation*}
Thus $f$ is not saturated. A saturated extension would itself be
an AMLE; uniqueness therefore rules out all saturated extensions
of this smooth datum and shows why Theorem~\ref{thm:domain}
requires $\dim Y\geq2$.
\end{remark}

Comparison on unbounded sets also prevents scalar extensions from
exceeding any finite bounds of the prescribed data. The distinction
can already be seen on the real line: $w(x)=x$ is an optimal
extension of $u(a)=a$ from $A=(-\infty,0]$ and satisfies comparison
on every bounded open set, yet $w(x)>\sup u(A)=0$ for $x>0$.
The missing comparison is on the ray $(0,\infty)$.
For vector-valued AMLEs, the corresponding requirement that the
image lie in the closed convex hull of the data fails, even in
the approximation sense stated below.

\begin{proposition}[A maximum principle for real-valued AMLEs]\label{prop:maximum}
Let $X$ be a Hilbert space, let $A\subset X$ be nonempty and
closed, let $u\colon A\to\R$ be Lipschitz, and let
$\widetilde u\colon X\to\R$ be an AMLE of $u$. Then
\begin{equation*}
 \inf\{u(a)\mid a\in A\}\leq\widetilde u(x)\leq\sup\{u(a)\mid a\in A\}
 \text{ for every }x\in X,
\end{equation*}
where an infinite bound is vacuous. By contrast, let
$X=Y=\R^2$, $A=\{0,e_1\}$, $u(0)=0$, and $u(e_1)=e_1$. Then for
every optimal extension $v_0$ of $u$, every compact $Q\subset X$
and every $\varepsilon>0$, there is an AMLE of $u$ within
$\varepsilon$ of $v_0$ on $Q$ whose image is not contained in
$\cl(\conv u(A))=[0,e_1]$.
\end{proposition}
\begin{proof}
Suppose $M=\sup\{u(a)\mid a\in A\}<\infty$ and set
$V=\{x\in X\mid\widetilde u(x)>M\}$. If $V$ were nonempty,
it would be a proper open subset of the connected space $X$,
with nonempty boundary and $\widetilde u=M$ on $\partial V$.
Comparison would make $\widetilde u$ constant on $V$, and
continuity at a boundary point would make this constant $M$,
contrary to the definition of $V$. Hence $\widetilde u\leq M$.
Apply the same argument to $-\widetilde u$ for the lower bound.

For the vector-valued example, perturb $v_0$ towards the identity:
\begin{equation*}
 v_t=(1-t)v_0+t\operatorname{id}_{\R^2},\qquad 0\leq t\leq1.
\end{equation*}
This preserves the prescribed values and the optimal Lipschitz
constant, while $v_t\to v_0$ uniformly on $Q$ as $t\to0$.
The second coordinate of $v_t(e_2)$ is an
affine function of $t$ that equals one at $t=1$, so it vanishes
for at most one $t$. Choose $t>0$ with nonzero second coordinate
and small enough that
$\|v_t(x)-v_0(x)\|<\varepsilon/2$ for all $x\in Q$, and put
$d=\dist(v_t(e_2),[0,e_1])>0$.
Theorem~\ref{thm:domain} on the connected domain
$\R^2\setminus A$ gives an AMLE approximating $v_t$ on
$Q\cup\{e_2\}$ with error less than
$\min\{\varepsilon/2,d/2\}$. The resulting AMLE is within
$\varepsilon$ of $v_0$ on $Q$ and still maps $e_2$ outside
$[0,e_1]$.
\end{proof}

Density gives a direct way to test closure under uniform
convergence. On a bounded finite-dimensional domain, it suffices
to find one optimal extension that fails comparison on a smaller
open set. The following example does this for a fixed smooth
boundary datum on a disk.

\begin{example}[Failure of closure under uniform convergence]\label{ex:uniform-nonclosure}
Identify $\R^2$ with $\mathbb C$, let
$\mathbb D=\{z\in\mathbb C\mid |z|<1\}$, and prescribe
\begin{equation*}
 u(z)=\tfrac12 z^2\text{ for }z\in\partial\mathbb D.
\end{equation*}
Up to scaling, this is the datum used in
\cite[Example~3.1]{SheffieldSmart} to demonstrate nonuniqueness.
We have $\Lip(u,\partial\mathbb D)=1$, since
$|u(z)-u(w)|/|z-w|=|z+w|/2$ for distinct boundary points.

Adjoin the values
\begin{equation*}
 v(z)=0\quad\text{when }|z|=\tfrac12,
 \qquad v(0)=\tfrac12.
\end{equation*}
Together with $u$, these data are $1$-Lipschitz: between the
two circles the image distance is $1/2$ and the source distance
is at least $1/2$, while pairs involving the origin have image
distance at most their source distance. Kirszbraun's theorem
therefore gives an optimal extension $v_0$ to $\cl\mathbb D$.
It is not absolutely minimising: on $B_{1/2}(0)$ its boundary
values are all zero, whereas $v_0(0)=1/2$.

Theorem~\ref{thm:domain}, applied with $Q=\cl\mathbb D$,
provides AMLEs $\widetilde u_j$ of the same datum with
$\|\widetilde u_j-v_0\|_{\infty}<1/j$.
Thus the AMLEs of this fixed smooth boundary datum are not closed
in the uniform topology.
\end{example}

Two questions about the approximation theorem remain. The first
is whether the approximation can be uniform on $\cl\Omega$ when
$\Omega$ is bounded but $X$ is infinite-dimensional; the proof
uses finite nets of compact sets, and hence does not provide such
uniform control on a general bounded domain. The second is
whether, for finite-dimensional $X$ and bounded $\Omega$, the
AMLEs form a residual subset of the optimal extensions, as the
length-preserving extensions in \cite{KSS} do in their setting;
the density theorem alone does not establish this, since
saturation requires comparison on the boundaries of all open sets
and its expression as a $G_\delta$ condition is not evident.
For separable sources, both questions can be answered
affirmatively by a Baire category argument on planar sections,
which combines the reflection estimate
\eqref{eq:quadratic-actual-margin}, Lemma~\ref{lem:nonseparation}
and Brehm's theorem in the plane: the saturated extensions then
contain a $G_\delta$ subset, for the topology of uniform
convergence on compact sets, that is dense for uniform
convergence on bounded sets. We do not pursue this here, as that
argument yields none of the structure in
Remark~\ref{rem:structure}. For nonseparable sources, uniform
approximation on bounded domains remains open.

\section{Linear absolutely minimising extensions}\label{sec:linear}

A linear map $T$ has Lipschitz constant $\|T\|$ on every
nonempty open set. Absolute minimality therefore asks whether
the boundary of every admissible comparison set also detects
this norm. The criterion below requires planes on which $T$
nearly attains its norm, uniformly over all unit vectors of the
plane. If no such planes exist, a hyperplane on which the norm
drops provides a halfspace that violates comparison.

\begin{proposition}[A linear characterisation]\label{prop:linear}
Let $T\colon X\to Y$ be a bounded linear map between Hilbert
spaces, and put $L=\|T\|>0$. The following are equivalent:
\begin{enumerate}
\item for every $\varepsilon>0$, there is a two-dimensional
linear subspace $E\subset X$ such that
\begin{equation*}
 \|Te\|\geq(L-\varepsilon)\|e\|
 \text{ for every }e\in E;
\end{equation*}
\item for every closed set $A\subset X$ with $\Lip(T,A)=L$,
the map $T$ is an AMLE of $T|_A$.
\end{enumerate}
When these conditions hold, every such set $A$ satisfies
\begin{equation*}
 \Lip(T,V)=\Lip(T,\partial V)=L
 \text{ for every nonempty open }V\subset X\setminus A.
\end{equation*}
\end{proposition}
\begin{proof}
Assume the first condition and let $A\subset X$ be closed with
$\Lip(T,A)=L$. Linearity gives $\Lip(T,V)=L$ on every
nonempty open set. Suppose some nonempty open
$V\subset X\setminus A$ has
$k=\Lip(T,\partial V)<L$, and choose a two-dimensional subspace
$E$ on which $\|Te\|\geq m\|e\|$ for some $m>k$.

Put $S=\partial V$. Each affine plane parallel to $E$
meets $S$ in at most one point: for any two points $p,q$ of
such an intersection,
\begin{equation*}
 m\|p-q\|\leq\|T(p-q)\|\leq k\|p-q\|,
\end{equation*}
and $m>k$ forces $p=q$. Lemma~\ref{lem:fibre-nonseparation}
with $O=X$ shows that $X\setminus S$ is connected.
The set $V$ is nonempty, relatively open and closed in
$X\setminus S$, so $V=X\setminus S$. Thus $A\subset S$,
contradicting $\Lip(T,A)=L>k$.

Conversely, suppose the first condition fails. Choose $0<c<L$
such that no two-dimensional subspace satisfies
$\|Te\|\geq c\|e\|$ for all of its vectors, and put
\begin{equation*}
 q(z)=\|Tz\|^2-c^2\|z\|^2.
\end{equation*}
Since $c<\|T\|$, there is an $x\in X$ with $q(x)>0$.
Define
\begin{equation*}
 \ell(y)=\langle Tx,Ty\rangle-c^2\langle x,y\rangle,
 \qquad F=\ker\ell.
\end{equation*}
The functional $\ell$ is nonzero because $\ell(x)=q(x)>0$.
For $y\in F$ and $s,t\in\R$, the cross term vanishes:
\begin{equation*}
 q(sx+ty)=s^2q(x)+t^2q(y).
\end{equation*}
If $q(y)>0$, this would make $q$ positive on the plane
$\Span\{x,y\}$, contrary to the choice of $c$.
Thus $\|T|_F\|\leq c$.
Now take $A=\{\ell\leq0\}$ and $V=\{\ell>0\}$.
Both contain nonempty open sets, so
\begin{equation*}
 \Lip(T,A)=\Lip(T,V)=L>c\geq\Lip(T,F)
 =\Lip(T,\partial V).
\end{equation*}
Hence the second condition fails.
\end{proof}

In finite dimensions, the criterion says that the largest singular
value of $T$ has multiplicity at least two. In infinite dimensions,
the norm need not be attained; planes with lower bounds approaching
$\|T\|$ suffice.

Unbounded comparison sets are essential to this characterisation.
Every bounded linear map satisfies comparison on bounded open sets.
Indeed, fix $x\in V$ and a unit vector $e$. If $V$ is bounded,
the component of $\{t\in\R\mid x+te\in V\}$ containing zero
is a bounded interval $(a,b)$, whose endpoints give two boundary
points with
\begin{equation*}
 \frac{\|T(x+be)-T(x+ae)\|}{\|(x+be)-(x+ae)\|}=\|Te\|.
\end{equation*}
Taking the supremum over $e$ gives
$\Lip(T,\partial V)=\|T\|=\Lip(T,V)$.

For complex-linear maps, a complex line through a vector where
the norm is nearly attained supplies a real plane of the required
type. On bounded planar domains, the affine maps
$z\mapsto az+b$ are also tight in the sense of Sheffield and Smart
\cite[Theorem~1.6]{SheffieldSmart}. The following consequence of
Proposition~\ref{prop:linear} applies to arbitrary complex Hilbert spaces.

\begin{corollary}[Complex-linear maps]\label{cor:complex-linear}
Let $T\colon X\to Y$ be a bounded complex-linear map between
complex Hilbert spaces. For every closed set $A\subset X$
with $\Lip(T,A)=\|T\|$, the map $T$ is an AMLE of $T|_A$.
\end{corollary}
\begin{proof}
The zero operator is constant. Otherwise, view the spaces as real
Hilbert spaces. For any unit vector $x\in X$, complex linearity
gives
\begin{equation*}
 \|Te\|=\|Tx\|\|e\|
 \text{ for every }e\in\Span_{\R}\{x,ix\}.
\end{equation*}
Choosing $x$ with $\|Tx\|$ arbitrarily close to $\|T\|$
verifies the first condition of Proposition~\ref{prop:linear}.
\end{proof}

The use of approximate norm attainment matters. An AMLE can have
Lipschitz constant one on every nonempty comparison set even when
no pair of distinct points attains Lipschitz ratio one.

\Needspace{12\baselineskip}
\begin{example}[An AMLE with singleton leaves]\label{ex:strict-linear}
Let $(e_j)_{j\geq1}$ be the standard orthonormal basis of
$X=Y=\ell^2$, and define
\begin{equation*}
 Te_j=\frac{j}{j+1}e_j,
 \text{ and }A=\{0\}\cup\{e_j/j\mid j\geq1\}.
\end{equation*}
The prescribed set $A$ and its image $T(A)$ are compact and
have infinite-dimensional affine span. The ratios for the pairs
$0,e_j/j$ tend to one, so $\Lip(T,A)=\|T\|=1$. The span of
any two sufficiently high coordinate vectors satisfies the
hypothesis of Proposition~\ref{prop:linear}. Thus $T$ is an
AMLE of $T|_A$, although
\begin{equation*}
 \|Tx-Ty\|<\|x-y\|\text{ whenever }x\ne y.
\end{equation*}
Every leaf of this absolutely minimising map is therefore a
singleton.
\end{example}

\Needspace{11\baselineskip}
The final construction extends prescribed linear data by completing
the operator to a scalar multiple of a coisometry. The operator
$D$ below specifies the missing term in the identity
$T_0T_0^*+D^2=L^2\Id_Y$. An isometric copy of the closure of its range in the
unused source directions is enough to realise this completion.

\begin{corollary}[Coisometric completion]\label{cor:completion}
Let $E$ be a closed linear subspace of a Hilbert space $X$, let
$Y$ be a Hilbert space with $\dim Y\geq2$, and let
$T_0\colon E\to Y$ be bounded and linear, with $L=\|T_0\|>0$.
Define
\begin{equation*}
 D=(L^2\Id_Y-T_0T_0^*)^{1/2},
 \text{ and }K=\cl(\operatorname{ran}D).
\end{equation*}
Suppose that there is a linear isometric embedding
$J\colon K\to E^\perp$. If $A\subset E$ is closed and
$\Lip(T_0,A)=L$, then
\begin{equation*}
 \widetilde u(e+z)=T_0e+DJ^*z
 \text{ for }e\in E,\ z\in E^\perp
\end{equation*}
is an AMLE of $T_0|_A$.
\end{corollary}
\begin{proof}
Since $D(Y)\subset K$ and $J^*J=\Id_K$, the displayed extension
satisfies
\begin{equation*}
 \widetilde u\widetilde u^*
 =T_0T_0^*+DJ^*JD=T_0T_0^*+D^2=L^2\Id_Y.
\end{equation*}
Thus $\widetilde u/L$ is a coisometry onto $Y$: it has norm
one and is isometric on $(\ker\widetilde u)^\perp$, which
has dimension $\dim Y\geq2$. Proposition~\ref{prop:linear}
therefore applies.
\end{proof}

\section*{Declaration on the use of artificial intelligence}
Large language models were used extensively in the preparation of
this paper. ChatGPT (OpenAI) and Claude (Anthropic) were used to
explore and develop the arguments, to check the proofs, to draft
and revise the text, to prepare the figures, and to compile and
check bibliographic information. The author posed the problem,
directed the work, verified every definition, statement and proof,
and takes full responsibility for the content.


\begin{thebibliography}{99}

\bibitem{AkopyanTarasov}
A.~V.~Akopyan and A.~S.~Tarasov,
\emph{A constructive proof of Kirszbraun's theorem},
Math. Notes \textbf{84} (2008), no.~5, 725--728.
DOI: \doi{10.1134/S000143460811014X}.

\bibitem{Aronsson}
G.~Aronsson, \emph{Extension of functions satisfying Lipschitz conditions},
Ark. Mat. \textbf{6} (1967), 551--561.
DOI: \doi{10.1007/BF02591928}.

\bibitem{ACJ}
G.~Aronsson, M.~G.~Crandall, and P.~Juutinen,
\emph{A tour of the theory of absolutely minimizing functions},
Bull. Amer. Math. Soc. (N.S.) \textbf{41} (2004), no.~4, 439--505.
DOI: \doi{10.1090/S0273-0979-04-01035-3}.

\bibitem{AzagraLeGruyerMudarra}
D.~Azagra, E.~Le Gruyer, and C.~Mudarra,
\emph{Kirszbraun's theorem via an explicit formula},
Canad. Math. Bull. \textbf{64} (2021), no.~1, 142--153.
DOI: \doi{10.4153/S0008439520000314}.

\bibitem{BarronJensenWang}
E.~N.~Barron, R.~R.~Jensen, and C.~Y.~Wang,
\emph{The Euler equation and absolute minimizers of $L^\infty$
functionals}, Arch. Ration. Mech. Anal. \textbf{157} (2001),
no.~4, 255--283.
DOI: \doi{10.1007/PL00004239}.

\bibitem{Brehm}
U.~Brehm, \emph{Extensions of distance reducing mappings to piecewise
congruent mappings on $\mathbb{R}^{m}$},
J. Geom. \textbf{16} (1981), no.~1, 187--193.
DOI: \doi{10.1007/BF01917587}.

\bibitem{CiosmakContinuity}
K.~J.~Ciosmak, \emph{Continuity of extensions of Lipschitz maps},
Israel J. Math. \textbf{245} (2021), no.~2, 567--588.
DOI: \doi{10.1007/s11856-021-2215-0}.

\bibitem{CiosmakLeaves}
K.~J.~Ciosmak, \emph{Leaves decompositions in Euclidean spaces},
J. Math. Pures Appl. (9) \textbf{154} (2021), 212--244.
DOI: \doi{10.1016/j.matpur.2021.08.003}.

\bibitem{CiosmakMonotone}
K.~J.~Ciosmak, \emph{Continuity of extensions of Lipschitz maps and
of monotone maps}, J. Lond. Math. Soc. (2) \textbf{110} (2024),
no.~5, Paper No.~e70014, 24~pp.
DOI: \doi{10.1112/jlms.70014}.

\bibitem{CiosmakUniform2026}
K.~J.~Ciosmak, \emph{Kirszbraun extensions preserving uniform distance
in Hilbert spaces}, preprint (2026), arXiv:2607.17672.
DOI: \doi{10.48550/arXiv.2607.17672}.

\bibitem{CEG}
M.~G.~Crandall, L.~C.~Evans, and R.~F.~Gariepy,
\emph{Optimal Lipschitz extensions and the infinity Laplacian},
Calc. Var. Partial Differential Equations \textbf{13} (2001),
no.~2, 123--139.
DOI: \doi{10.1007/s005260000065}.

\bibitem{HirnLeGruyer}
M.~J.~Hirn and E.~Y.~Le Gruyer,
\emph{A general theorem of existence of quasi absolutely minimal
Lipschitz extensions}, Math. Ann. \textbf{359} (2014),
nos.~3--4, 595--628.
DOI: \doi{10.1007/s00208-013-1003-5}.

\bibitem{Jensen}
R.~Jensen, \emph{Uniqueness of Lipschitz extensions: minimizing
the sup norm of the gradient}, Arch. Rational Mech. Anal.
\textbf{123} (1993), no.~1, 51--74.
DOI: \doi{10.1007/BF00386368}.

\bibitem{Juutinen}
P.~Juutinen, \emph{Absolutely minimizing Lipschitz extensions on
a metric space}, Ann. Acad. Sci. Fenn. Math. \textbf{27} (2002),
no.~1, 57--67.

\bibitem{Katzourakis}
N.~Katzourakis, \emph{$L^\infty$ variational problems for maps and
the Aronsson PDE system}, J. Differential Equations \textbf{253}
(2012), no.~7, 2123--2139.
DOI: \doi{10.1016/j.jde.2012.05.012}.

\bibitem{KSS}
B.~Kirchheim, E.~Spadaro, and L.~Sz\'ekelyhidi Jr.,
\emph{Equidimensional isometric maps},
Comment. Math. Helv. \textbf{90} (2015), no.~4, 761--798.
DOI: \doi{10.4171/CMH/370}.

\bibitem{Kirszbraun}
M.~D.~Kirszbraun,
\emph{\"Uber die zusammenziehende und Lipschitzsche Transformationen},
Fund. Math. \textbf{22} (1934), 77--108.
DOI: \doi{10.4064/fm-22-1-77-108}.

\bibitem{Milman}
V.~A.~Milman, \emph{Absolutely minimal extensions of functions on
metric spaces}, Sb. Math. \textbf{190} (1999), no.~6, 859--885.
DOI: \doi{10.1070/sm1999v190n06ABEH000409}.

\bibitem{NaorSheffield}
A.~Naor and S.~Sheffield,
\emph{Absolutely minimal Lipschitz extension of tree-valued mappings},
Math. Ann. \textbf{354} (2012), no.~3, 1049--1078.
DOI: \doi{10.1007/s00208-011-0753-1}.

\bibitem{PSSW}
Y.~Peres, O.~Schramm, S.~Sheffield, and D.~B.~Wilson,
\emph{Tug-of-war and the infinity Laplacian},
J. Amer. Math. Soc. \textbf{22} (2009), no.~1, 167--210.
DOI: \doi{10.1090/S0894-0347-08-00606-1}.

\bibitem{SheffieldSmart}
S.~Sheffield and C.~K.~Smart,
\emph{Vector-valued optimal Lipschitz extensions},
Comm. Pure Appl. Math. \textbf{65} (2012), no.~1, 128--154.
DOI: \doi{10.1002/cpa.20391}.

\bibitem{Valentine}
F.~A.~Valentine, \emph{A Lipschitz condition preserving extension
for a vector function}, Amer. J. Math. \textbf{67} (1945),
no.~1, 83--93.
DOI: \doi{10.2307/2371917}.

\end{thebibliography}
\end{document}